\documentclass[11pt]{amsart}
\usepackage{amsfonts}
\usepackage{indentfirst,amsthm,bm,amsmath}
\usepackage{hyperref}
\usepackage{graphicx}
\newtheorem{theorem}{Theorem}[section]
\newtheorem{corollary}[theorem]{Corollary}
\newtheorem{lemma}[theorem]{Lemma}

\newtheorem{remark}[theorem]{Remark}
\newtheorem{example}[theorem]{Example}

\newcommand{\N}{\mathbb N}
\newcommand{\Q}{\mathbb Q}

\newcommand{\ds}{\displaystyle}
\newcommand{\C}{\mathbb{C}}
\newcommand{\R}{\mathbb{R}}
\newcommand{\ug}{\underline{g}}
\newcommand{\og}{\overline{g}}

\newcommand{\of}{\overline{f}}
\newcommand{\uh}{\underline{h}}

\begin{document}

\title[]{On the Fermat-type partial differential-difference equations on $\C^n$}

\author[T. B. Cao]{Tingbin Cao*}
\address[Tingbin Cao]{Department of Mathematics\\ Nanchang University\\ Nanchang 330031\\
Jiangxi\\ China}
\email{tbcao@ncu.edu.cn}
\thanks{* This paper was supported by the National Natural Science Foundation of China (\#11871260, \#11771090) and the Jiangxi Provincial Natural Science Foundation of China (\# 20232ACB201005).}
\author[J. Wang]{Jun Wang*}
\address[Jun Wang]{School of Mathematics\\ Fudan University\\
Shanghai 200433\\ China}
\email{majwang@fudan.edu.cn}
\author[Z. Ye]{Zhuan Ye}
\address[Zhuan Ye]{Department of Mathematics and Statistic\\ University of Northern Carolina\\ 601 South College Road\\ Wilmington\\ NC 28403-5970\\ USA}
\email{yez@uncw.edu}

\subjclass[2010]{Primary 35B08; 32W50; 39A45; Secondary 32H30; 35A09}

\keywords{Entire functions; Fermat-type equation; Nevanlinna theory; Partial differential-difference equation}

\date{}

\begin{abstract}
Assume that $n$ is a positive integer, $p_j$ ($j=1,2, \cdots, 6)$ are polynomials, $p$ is an irreducible polynomial, and
$f$ is an entire function on $\C^n$. Let
$L(f)=\sum_{j=1}^s q_{t_j}f_{z_{t_j}}$ and $\of(z)=f(z_{1}+c_{1}, \ldots, z_{n}+c_{n})$, where $q_{t_j}$ ($j=1,2, \cdots, s\le n$) are non-zero polynomials on $\C^n$ and  $c=(c_{1}, \ldots, c_{n})\in \mathbb{C}^{n}\setminus\{0\}$.
We show the structures of all entire solutions to  the non-linear partial differential-difference equation
$$(p_{1} L(f)+p_{2}{\of}+p_5 f)^{2}+(p_{3}L(f)+p_{4}{\of}+p_6 f)^{2}=p.$$
The partial differential-difference equation is called a Fermat-type partial differential-difference equation (PDDE).
Further, we find many sufficient conditions and/or necessary conditions for the existence, as well as the concrete representations, of entire solutions to the Fermat-type PDDE.
We also demonstrate several examples on $\mathbb{C}^2$ with non-constant coefficients to verify that all representations in our theorems exist and are accurate and that the entire solutions to the Fermat-type PDDEs could have finite or infinite growth order. Our theorems unify and extend previous results (see, e.g., \cite{cao-xu-2020},
\cite{chang-li-2012}, \cite{khavison-1995}, \cite{li-2007}, \cite{zhangXiaoFang-2022}).
\end{abstract}

\maketitle
\tableofcontents

\section{Introduction}
Fermat's last theorem \cite{wiles-1995} states that the equation $x^{m}+y^{m}=1$, where integer $m\ge 2$  and $x, y, \in \Q$, does not admit nontrivial solutions over the field of rational numbers when $m\geq 3$, and admits nontrivial rational solutions when $m=2.$
In 1927, Montel \cite{montel-1927} studied the solutions to the functional equation $f^{m}+g^{m}=1$, which is called as  a Fermat-type functional equation, over the field of meromorphic functions on $\C$.
Since then, the studies of the Fermat-type functional equation 
have been extended in many directions such as, $af^m+bf^k=c$ (where $a,b,c$ are entire), $u_{z_1}^2+u_{z_2}^2 =p_1e^{p_2}$ (where $p_1$ and $p_2$ are polynomials on $\C^2$), $f^m(z)+f^k(z+z_0)=1$ (where $z_0\in \C^n$ is fixed), to list a few.
In 1990's, Khavinson \cite{khavison-1995} proved  the only entire solutions to  the Fermat-type of PDE $u_{z_1}^2+u_{z_2}^2 =1$ on $\C^2$ are linear polynomials.
Let $u$ and $F$ be entire on $\C^2$ and $F$ does not have a linear factor. In 1997, Hemmati \cite{hemmati1997} proved that if $F(u_{z_1}, u_{z_2})=0$, then $u$ is a linear function.
In 2000's, Saleeby\cite{saleeby-1999}, B. Q.  Li and his collaborators published a series of
articles about the entire or meromorphic solutions to the Fermat-type PDE, as well as  the Fermat-type functional equations (e.g. \cite{chang-li-2012, li-2005, li-2007, li-2012,  li-2014, li-ye-2008}).
In 2012, Liu-Cao-Cao \cite{liu-cao-cao-2012} introduced the Fermat-type partial differential-difference equations (PDDE). Recently, the studies of the Fermat-type PDDE in $\C^n$ are an active research topic
and have frequently appeared in the vast of literatures (e.g.  \cite{ahamed-Allu2023}, \cite{cao-xu-2020, han-lv-2019, huWangWu-2022, li-2014, lv-li-2019, magTal-2003, xu-cao-2020, xu-liu-li-2020, xuTuWang-2021, zhangXiaoFang-2022}).

It is known that  Fermat's last theorem is connected to elliptic curves and the Fermat-type functional equations are connected to holomorphic curves. The Fermat-type of PDDE is a generalization of the Fermat-type PDE, which is an extension of eiconal-type non-linear PDE. By a little calculation, the eiconal-type equation turns to be a well-known Monge-Ampere equations. It is known that both real/complex eiconal equations and complex Monge-Ampere equations have a lot of applications in the approximations of wave propagation and complex geometry.

The main purpose of the paper is to systematically study Fermat-type PDDEs on $\C^n$ in a rather broad setting.
We reveal the structures of all entire solutions to the Fermat-type PDDEs and find many sufficient conditions and/or necessary conditions for the existence, as well as concrete representations, of entire solutions. For some special Fermat-type PDDEs, we are able to solve the PDDEs completely.
We also demonstrate several examples on $\mathbb{C}^2$ with non-constant coefficients to verify that all representations in our theorems exist and are accurate. Further, these examples show that the entire solutions to the Fermat-type PDDEs could have finite  or infinite growth order.
Many results from previous publications (see, e.g., \cite{cao-xu-2020},
\cite{chang-li-2012}, \cite{khavison-1995}, \cite{li-2007}, \cite{zhangXiaoFang-2022}) are our corollaries (see, e.g., Corollaries \ref{Cor3.8} and \ref{Cor4.2}). We also correct some errors in some previous publications. To limit the length of the paper, we only show several important corollaries although we could obtain a lot of other results with slightly changes of coefficients of the PDDEs. The readers could generate more corollaries as they need.

We are the first to convert the Fermat-type PDDEs to a matrix equation in PDDE theory and to consider all possible factorizations of the PDDEs in question, which are shown in Lemma \ref{equiv}. The matrix method used in the paper makes the complicated calculations to be easier and clearer. The advantage of the matrix method is convincingly demonstrated throughout the paper, especially, in the proof of Theorem \ref{T6.1}.
We believe that this matrix method is a powerful tool to study a single or system of Fermat-type PDDEs on $\C^n$ in the future. The proofs and calculations of our results and examples involve many clever techniques.
It is quite common in the study of Fermat-type difference equations and PDDEs that one needs an assumption that  entire solutions are of finite order due to the use of the logarithmic difference lemma. The methods we utilize in our proofs have successfully eliminated the assumption
since we have adopted the methods from partial differential equation theory, Nevalinna theory, function theory, as well as difference equation theory.

\par

In order to state our results, we introduce following notations. For any meromorphic function $f(z):\mathbb{C}^{n}\to \mathbb{C}\cup\{\infty\}$ and any non-zero constant $c=(c_{1}, \ldots, c_{n})\in \mathbb{C}^{n}\setminus\{0\}$, set
$$f=f(z_{1}, \ldots, z_{n}), \quad\overline{f}=f(z_{1}+c_{1}, \ldots, z_{n}+c_{n}),$$ $$ \underline{f}=f(z_{1}-c_{1}, \ldots, z_{n}-c_{n}), \quad
f_{z_{j}}=\frac{\partial f(z_{1}, \ldots, z_{n})}{\partial z_{j}},$$
$$\overline{f}_{z_{j}}=\frac{\partial f(z_{1}+c_{1}, \ldots, z_{n}+c_{n})}{\partial z_{j}},\quad \underline{f}_{z_{j}}=\frac{\partial f(z_{1}-c_{1}, \ldots, z_{n}-c_{n})}{\partial z_{j}}.$$
The reader should be able to distinct it from the common conjugate operations for complex numbers in the context.

In the sequel, we always assume that  $p_1,p_2, \cdots, p_6$ are polynomials and $p$ is a non-zero irreducible polynomial on $\mathbb{C}^n$. We denote the degree
of a polynomial $b$ in $z_j$ by $\deg_{z_j}b$ and a non-zero complex number by $c_*$ although each of its appearance  may have a different value.  Further, set $k=\pm 1$,
\begin{align*}
a_1 & =kp_1-ip_{3},\qquad\qquad\quad  a_2 =kp_{1}+ip_3,\\
b_1 & = kp_2 -ip_{4},\qquad\qquad\quad     b_2 = kp_{2}+ip_4, \\
d_1 &=p_2p_6-p_4p_5,\quad\quad\quad\quad d_2 = p_3p_5-p_1p_6,\\
D&\stackrel{def}{=}p_1p_4-p_2p_3=k(a_1b_2 -a_2b_1)/(2i).&
\end{align*}
These definitions are important in our theorems and proofs, and are reserved for the entire paper.\par

Let $s\le n$.  For any $t_j\in \N$ with $t_1<t_2<\cdots<t_s$, we define $L(f)=\sum_{j=1}^s q_{t_j}f_{z_{t_j}}$, where $q_{t_j}$ ($j=1,2, \cdots, s$) are non-zero polynomials on $\C^n$.
In the sequel, for the brevity, we simply write $t_j$ as $j$ and
$$
L(f) =\sum_{j=1}^sq_jf_{z_j}.
$$


In this paper, we systematically investigate the non-trivial entire solutions to the Fermat-type PDDE
\begin{equation}\label{E1.1}(p_{1} L(f)+p_{2}\overline{f}+p_5 f)^{2}+(p_{3}L(f)+p_{4}\overline{f}+p_6 f)^{2}=p\end{equation}
by considering the following four cases:\par
\begin{enumerate}
\item[I.]\,\, $D\not\equiv 0$ and $d_2\equiv 0;$ (Theorem \ref{T1})
\item[II.]\,\, $D\equiv 0$ and $d_2\not\equiv 0$; (Theorem\ref{T17})
\item[III.]\,\, $D\equiv 0,$ $d_1\not\equiv 0$ and $d_2\equiv 0$; (Theorem \ref{T1.3})
\item[IV.]\,\, $D\not\equiv 0$ and $d_2\not\equiv 0;$ (Theorem \ref{T6.1} when all $q_j$ in $L(f)$ are constants).
\end{enumerate}
For the remained case that $D\equiv d_1\equiv d_2\equiv 0$, the vectors $(p_1, p_2, p_5)$ and $(p_3, p_4, p_6)$ are linearly dependent. Hence, $$p_{3} L(f)+p_{4}{\of}+p_6 f=c_*(p_{1} L(f)+p_{2}{\of}+p_5 f),$$ for some $c_*\neq 0$. Then, \eqref{E1.1} is reduced to the simpler PDDE $$(1+c_*^2)(p_{1} L(f)+p_{2}{\of}+p_5 f)^2=p,$$ which
is  reduced further to
 $$p_{1} L(f)+p_{2}\of+p_5 f=constant$$
due to the hypothesis that $p$ is irreducible. The kind of linear PDDEs is not Fermat-type and so,  is not considered in the paper although we could obtain some of the solutions easily. \par

The rest of our paper is organized as follows. Section $2$ contains all definitions and theorems in Nevanlinna theory needed in this paper. It also includes our lemmas, where  Lemmas \ref{equiv}, \ref{new3} and \ref{new2} play an important role in the proofs of our theorems. Sections $3-6$ deal with Eq.\eqref{E1.1} according to Cases I-IV respectively.

\section{Preliminaries and Lemmas}

In this section, for the convenience of the reader, we recall briefly some basic definitions and theorems from Nevanlinna theory for meromorphic functions in several complex variables, which can be found in many literatures, e.g. see: \cite{griffiths-book-1976, hu-li-yang-book-2003, nouguchiWinkelmannBook-2010, ru-book-2001, shabat-book-1992, yezhuan-1995, yezhuan-1996}. Let $\mathbb{C}^{n}$ be $n$-dimensional complex Euclidean space with coordinates $z=(z_{1}, \ldots, z_{n}).$ Set $$\|z\|=(|z_{1}|^{2}+|z_{2}|^{2}+\ldots+|z_{n}|^{2})^{\frac{1}{2}}$$
and further, for $r>0,$
$$B_{n}(r)=\{z\in\mathbb{C}^{n}: \|z\|<r\}\quad \mbox{and} \quad S_{n}(r)=\{z\in\mathbb{C}^{n}:
\|z\|=r\}.$$ Let $d=\partial+\overline{\partial}$ and
$d^{c}=\frac{i}{4\pi}(\overline{\partial}-\partial).$ Denote $$\upsilon_{n}(z)=(dd^{c}\log\|z\|^{2})^{n-1}$$  and $$\sigma_{n}(z)=d^{c}\log\|z\|^{2}\wedge\upsilon_{n}(z)$$ for $z\in\mathbb{C}^{n}\setminus\{0\}.$ Note that $dd^{c}\|z\|^{2}=\frac{i}{2\pi}(\sum_{j=1}^{n}dz_{j}\wedge d\overline{z}_{j})$ means the standard K\"{a}hler form on $\mathbb{C}^{n},$ and $dd^{c}\log\|z\|^{2}$ is the pull-back to $\mathbb{C}^{n}\setminus\{0\}$ of the Fubini-Study K\"{a}hler metric on  complex projective space $\mathbb{P}^{n-1}(\mathbb{C}).$ Then $\sigma_{n}(z)$ is a positive measure on $S_{n}(r)$ with the total measure one and $(dd^{c}\|z\|^{2})^{n}$ is Lebesgue measure on $\mathbb{C}^{n}$ normalized such that $B_{n}(r)$ has measure $r^{2n}.$ Moreover, when we define $dd^{c}\|z\|^{2}$ to $S_{n}(r),$ we get that $$dd^{c}\|z\|^{2}=r^{2}dd^{c}\log\|z\|^{2}\quad\mbox{and}\quad \int_{B_{n}(r)}(dd^{c}\log\|z\|^{2})^{n}=1.$$\par

Let $h$ be a nonzero holomorphic function on $\mathbb{C}^{n}$ (usually, called entire function). For $a\in
\mathbb{C}^{n},$ we can write $h$ as
$h(z)=\sum_{j=0}^{\infty}P_{j}(z-a),$ where the term $P_{j}(z)$ is
either identically zero or a homogeneous polynomial of degree $j.$
The number $\nu_{h}(a):=\min\{j: P_{j}\neq 0\}$ is said to be the
zero-multiplicity of $h$ at $a.$ Set $\mbox{supp}
\nu_{h}:=\overline{\{z\in\mathbb{C}^{n}: \nu_{h}(z)\neq 0\}}.$\par

Let $\varphi$ be a nonzero meromorphic function on $\mathbb{C}^{n}$ with reduced representation $\varphi=\frac{\varphi_{0}}{\varphi_{1}},$ where $\varphi_{0}$ and $\varphi_{1}$ are entire functions on $\mathbb{C}^{n}$ having no common zeros.  We
define $\nu_{\varphi}^{0}:=\nu_{\varphi_{0}}$
and $\nu_{\varphi}^{\infty}:=\nu_{\varphi_{1}},$ which are independent
of choices of $\varphi_{0}$ and $\varphi_{1}.$ \par

Let $K$ be a positive integer or $+\infty.$ For a divisor $\nu$ on
$\mathbb{C}^{n},$ we define the counting functions of $\nu$ as
follows. Set
\begin{equation*}
\nu^{K}(z)=\min\{\nu(z), K\}
\end{equation*} and
\begin{eqnarray*} n_{\nu}^{K}(t)=\left\{
                                    \begin{array}{ll}
                                      \int_{\mbox{\small{supp}}\nu\cap B_{n}(t)}\nu^{K}(z)\upsilon_{n}(z), & \hbox{if $n\geq 2;$} \\\\
                                      \sum_{|z|\leq t}\nu^{K}(z), & \hbox{if $n=1.$}
                                    \end{array}
                                  \right.
\end{eqnarray*}We define
\begin{equation*}
N_{\nu}^{K}(r)=N^{K}(r, \nu)=\int_{1}^{r}\frac{n^{K}(t)}{t^{2n-1}}dt\quad
(r>1).\end{equation*}

For a meromorphic function $\varphi$ on $\mathbb{C}^{n},$ we denote
$$N^{K}(r, \frac{1}{\varphi}):=N^{K}(r,  \nu_{\varphi}^{0}),\quad N^{K}(r, \varphi):=N^{K}(r, \nu_{\varphi}^{\infty}).$$
For brevity, we will omit
the superscript $K$ if $K=+\infty.$ Usually, we  use the notation $\overline{N}(r, \frac{1}{\varphi})$ instead of $N^{1}(r, \varphi).$ We have the following Jensen's formula:\par

$$N(r, \frac{1}{\varphi})-N(r, \varphi)=\int_{S_{n}(r)}\log|\varphi(z)|\sigma_{n}(z)-\int_{S_{n}(1)}\log|\varphi(z)|\sigma_{n}(z).
$$
\par

For any $a\in \mathbb{C}\cup\{\infty\}$, we define the proximity function $m(r, \frac{1}{\varphi-a})$ by

\begin{eqnarray*}m(r, \frac{1}{\varphi-a}):=\left\{
                                    \begin{array}{ll}
                                      \int_{S_{n}(r)}\log^{+}\frac{1}{|\varphi(z)-a|}\sigma_{n}(z), & \hbox{if $a\neq\infty;$} \\\\
                                      \int_{S_{n}(r)}\log^{+}|\varphi(z)|\sigma_{n}(z), & \hbox{if $a=\infty,$}
                                    \end{array}
                                  \right.
\end{eqnarray*}where $\log^{+}x=\max\{\log x, 0\}.$ The Nevanlinna's characteristic function of $\varphi$ is defined as
\begin{eqnarray*}
T(r, \varphi):=m(r, \varphi)+N(r, \varphi).
\end{eqnarray*}

Let $I=(i_{1}, \ldots, i_{n})$ be a multi-index in $(\mathbb{Z}^{+})^{n}$ with length $|I|=\sum_{j=1}^{n}i_{j}$, and let
$$\partial^{I}\varphi=\frac{\partial^{|I|}\varphi}{\partial z_{1}^{i_{1}}\cdots \partial z_{n}^{i_{n}}}.$$
\par

Now, we state three most important theorems in Nevanlinna theory, which are frequently used in our proofs.

{\bf The first main theorem.} Let $\varphi$ be a non-zero meromorphic function on $\mathbb{C}^{n}$ and $a\in \C$. Then
\begin{eqnarray*}T(r, \frac{1}{\varphi-a})=T(r, \varphi)+O(1)
\end{eqnarray*}

{\bf The second main theorem.} Let $\varphi$ be a non-zero meromorphic function on $\mathbb{C}^{n}$ and let
$a_{1},$ $\ldots,$ $a_{q}$ are $q \ge 3$ distinct complex values in the project complex plane $\mathbb{P}(\C)$.
Then
\begin{equation*}
(q-2)T(r, \varphi)<\sum_{j=1}^{q}{N}(r, \frac{1}{\varphi-a_{j}})+O(\log (rT(r, \varphi)))
\end{equation*}
holds for all $r\not\in E \subset \R^+,$ where $E$ is a set with finite Lebesgue measure on $\R^+$.

{\bf Logarithmic derivative lemma.}
 Let $\varphi$ be a non-constant meromorphic function on $\mathbb{C}^{n}$. Assume that $T(r_{0},\varphi)\geq e$ for some $r_{0}.$ Then
$$
m(r, \frac{\partial^{I}\varphi}{\varphi})=O(\log (rT(r, \varphi)))
$$
holds for all $r\geq r_{0}$ outside a set $E\subset (0, +\infty)$ of finite logarithmic measure, that is $\ds \int_{E}\frac{dt}{t}<\infty$, and any $I\in (\mathbb{Z}^{+})^{n}$.

It is known 
that a non-constant meromorphic function $\varphi$ on $\C^n$ is rational if and only if $T(r, \varphi)=O(\log r)$ and consequently, $f$ is transcendental if and only if $$\lim_{r\rightarrow\infty}\frac{T(r, \varphi)}{\log r}=+\infty.$$
The growth order, or simply, order of a meromorphic function $\varphi$ on $\mathbb{C}^{n}$ is defined by
$$
\rho(\varphi)=\limsup_{r\rightarrow\infty}\frac{\log^{+}T(r,\varphi)}{\log r}.
$$ \par
It is known (e.g. \cite[Proposition 4.2]{stoll-1968}) that $\varphi$, $\overline{\varphi}$ and $\underline{\varphi}$ have the same growth order.

Now, we are ready to prove our lemmas. Our first lemma, Lemma \ref{equiv}, is called an equivalence lemma, which is a
crucial starting point for us to prove all theorems in this paper.
\begin{lemma}\label{equiv}
Let $s$, $t$, and $u$ be three different operators on the space of entire functions on $\C^n$. Let $f,\eta $, and $\eta_i\, (i=1,2,\cdots,6)$ be entire functions
on $\C^n$ where $\eta$ is irreducible. If $\xi_0\stackrel{def}{=}\eta_1\eta_4-\eta_2\eta_3\not\equiv 0$, then $f$ is a solution to
\begin{align}\label{lem1eq}
\left(\eta_1s(f)+\eta_2t(f)+\eta_5u(f)\right)^2+\left(\eta_3s(f)+\eta_4t(f)+\eta_6u(f)\right)^2=\eta
\end{align}
if and only if there exists an entire function $g$ such that $f$ is a solution to
\begin{align}\label{lem2eq}
\left(\begin{array}{c}
    s(f)  \\ [5pt] t(f)
\end{array}\right)
= \frac{1}{2i\xi_0}\left(
  \begin{array}{cc}
    {-\zeta_1} & \eta \zeta_2 \\ [5pt] \tau_1 & -\eta \tau_2
  \end{array}
\right)\left(
  \begin{array}{c}
   e^{ig} \\  [5pt]e^{-ig}
  \end{array}
\right)
+
 \frac{1}{\xi_0} \left(
  \begin{array}{c}
  \xi_1 \\ [5pt] \xi_2
  \end{array}
\right)u(f),
\end{align}
where $\tau_1,\tau_2,\zeta_1,\zeta_2,\xi_1,\xi_2$ are defined by
\begin{align*}
\tau_1 & =k\eta_1-i\eta_{3},\qquad\zeta_1= k\eta_2 -i\eta_{4},\qquad\xi_1 =\eta_2\eta_6-\eta_4\eta_5,\\
\tau_2 &=k\eta_{1}+i\eta_3,\qquad\zeta_2 = k\eta_{2}+i\eta_4,\qquad \xi_2 = \eta_3\eta_5-\eta_1\eta_6.
\end{align*}
Further, if there exists another entire function $g_*$ such that $f$ is a soluton to \eqref{lem2eq}, then either $g+g_*$ or $g-g_*$ is constant.
\end{lemma}
\begin{proof}
Set $X=\eta_1s(f)+\eta_2t(f)+\eta_5u(f)$ and $Y=\eta_3s(f)+\eta_4t(f)+\eta_6u(f).$ Then,
 $$X^2+Y^{2}=(X+iY)(X-iY)=\eta.$$
Since $\eta$ is irreducible, there exists an entire function $g$ in $\mathbb{C}^{n}$ such that
$$\left\{
  \begin{array}{ll}
   X+iY=& e^{ig}, \\
    X-iY=& \eta e^{-ig},
  \end{array}
\right.
\text{or}\quad
\left\{
  \begin{array}{ll}
   X-iY=& e^{ig}, \\
    X+iY=& \eta e^{-ig}.
  \end{array}
\right.
$$
Thus, the solutions to these two systems of equations can be written as
 \begin{equation}\label{lem2eq-3}
\left(
  \begin{array}{c}
    X \\ [5pt] Y
  \end{array}
\right)=\frac{1}{2i}\left(
  \begin{array}{cc}
   i  &  i  \\ [5pt]
   k  &  -k \\
  \end{array}
\right) \left(
  \begin{array}{c}
  e^{ig} \\ [5pt] \eta e^{-ig}
  \end{array}
\right)
=\frac{1}{2i}\left(
  \begin{array}{cc}
   i  &  i\eta \\ [5pt]
   k  &  -k\eta \\
  \end{array}
\right) \left(
  \begin{array}{c}
  e^{ig} \\ [5pt]e^{-ig}
  \end{array}
\right),
\end{equation}
where $k=\pm 1.$
On the other hand, the definitions of $X$ and $Y$ give
 \begin{equation*}
\left(
  \begin{array}{c}
    X \\ [5pt] Y
  \end{array}
\right)=\left(
  \begin{array}{cc}
   \eta_1 &  \eta_2  \\ [5pt]
   \eta_3 &  \eta_4 \\
  \end{array}
\right) \left(
  \begin{array}{c}
  s(f) \\ [5pt] t(f)
  \end{array}
\right)
+
\left(
  \begin{array}{cc}
   0 &  \eta_5  \\[5pt]
   0 &  \eta_6 \\
  \end{array}
\right) \left(
  \begin{array}{c}
  0 \\ [5pt]u(f)
  \end{array}
\right)
\end{equation*}
It follows that
\begin{align*}
\left(
  \begin{array}{cc}
    \eta_1 & \eta_2 \\ [5pt]  \eta_3 & \eta_4
  \end{array}
\right)\left(\begin{array}{c}
     s(f) \\ [5pt] t(f)
\end{array}\right)
& =\frac{1}{2i}\left(
  \begin{array}{cc}
   i  &  i \eta \\ [5pt]
   k  &  -k\eta \\
  \end{array}
\right) \left(
  \begin{array}{c}
   e^{ig} \\ [5pt] e^{-ig}
  \end{array}
\right) \\
& \qquad -
\left(
  \begin{array}{cc}
   0 &  \eta_5  \\[5pt]
   0 &  \eta_6 \\
  \end{array}
\right) \left(
  \begin{array}{c}
  0 \\ [5pt]u(f)
  \end{array}
\right),
\end{align*}
which gives
\begin{eqnarray*}
\left(\begin{array}{c}
    s(f) \\ [5pt]t(f)
\end{array}\right)
  \hskip-.05in&= & \hskip-.05in \frac{1}{2i\xi_0}\left(
  \begin{array}{cc}
   \eta_{4} & -\eta_{2}  \\ [5pt]-\eta_{3} & \eta_{1}
  \end{array}
\right)\left(
  \begin{array}{cc}
   i  & i\eta \\[5pt]
   k &  -k\eta \\
  \end{array}
\right) \left(
  \begin{array}{c}
   e^{ig} \\  [5pt]e^{-ig}
  \end{array}
\right)\\ \nonumber
& \  & \qquad +
 \frac{1}{\xi_0}\left(
  \begin{array}{cc}
   -\eta_{4} & \eta_{2}  \\ [5pt] \eta_{3} & -\eta_{1}
  \end{array}
\right)
\left(
  \begin{array}{cc}
   0 &  \eta_5  \\[5pt]
   0 &  \eta_6 \\
  \end{array}
\right) \left(
  \begin{array}{c}
  0 \\ [5pt]u(f)
  \end{array}
\right)
\\ \nonumber
\hskip-.05in&=& \hskip-.05in\frac{1}{2i\xi_0}\left(
  \begin{array}{cc}
   -(k\eta_2-i\eta_{4}) & \eta(k\eta_{2}+i\eta_4)  \\ [5pt] k\eta_1-i\eta_{3} & -\eta(k\eta_{1}+i\eta_3)
  \end{array}
\right)\left(
  \begin{array}{c}
   e^{ig} \\ [5pt]  e^{-ig}
  \end{array}
\right) \\ \nonumber
& \  &  \qquad +
 \frac{1}{\xi_0}\left(
  \begin{array}{cc}
   0 &  -\eta_4\eta_5+\eta_2\eta_6  \\[5pt]
   0 &  \eta_3\eta_5-\eta_1\eta_6 \\
  \end{array}
\right) \left(
  \begin{array}{c}
  0 \\ [5pt]u(f)
  \end{array}
\right)
\\ \nonumber
\hskip-.05in&=&\hskip-.05in
 \frac{1}{2i\xi_0}\left(
  \begin{array}{cc}
   {-\zeta_1} & \eta{\zeta_2} \\ [5pt] \tau_1 & -\eta\tau_2
  \end{array}
\right)\left(
  \begin{array}{c}
   e^{ig} \\ [5pt] e^{-ig}
  \end{array}
\right)
+
 \frac{1}{\xi_0}\left(
  \begin{array}{cc}
   0 &  \xi_1  \\[5pt]
   0 &  \xi_2 \\
  \end{array}
\right) \left(
  \begin{array}{c}
  0 \\ [5pt]u(f)
  \end{array}
\right).
\end{eqnarray*}
Thus, the equations \eqref{lem1eq}  and \eqref{lem2eq}
are equivalent under the factorization $X^2+Y^2=(X+iY)(X-iY).$
\par Thus, we also need to consider another
factorization $X^2+Y^2=(Y+iX)(Y-iX).$  As in \eqref{lem2eq-3}, there exists an entire function $g_*$ such that
\begin{equation*}
\left(
  \begin{array}{c}
    Y \\ [5pt] X
  \end{array}
\right)
=\frac{1}{2i}\left(
  \begin{array}{cc}
   i  &  i\eta \\ [5pt]
   k_*  &  -k_*\eta \\
  \end{array}
\right) \left(
  \begin{array}{c}
  e^{ig_*} \\ [5pt]e^{-ig_*}
  \end{array}
\right),
\end{equation*}
where $k_*=\pm 1.$ This implies
\begin{equation*}\label{lem2eq-4}
\left(
  \begin{array}{c}
    X \\ [5pt] Y
  \end{array}
\right)
=\frac{1}{2i}\left(
  \begin{array}{cc}
   k_*  &  -k_*\eta \\ [5pt]
   i  &  i\eta \\
  \end{array}
\right) \left(
  \begin{array}{c}
  e^{ig_*} \\ [5pt]e^{-ig_*}
  \end{array}
\right),
\end{equation*}
Comparing this with \eqref{lem2eq-3} gives
\begin{align*}
\left(
  \begin{array}{cc}
   i  &  i\eta \\ [5pt]
   k  &  -k\eta \\
  \end{array}
\right)
\left(
  \begin{array}{c}
  e^{ig} \\ [5pt]e^{-ig}
  \end{array}
\right)
=\left(
  \begin{array}{cc}
   k_*  &  -k_*\eta \\ [5pt]
   i  &  i\eta \\
  \end{array}
\right) \left(
  \begin{array}{c}
  e^{ig_*} \\ [5pt]e^{-ig_*}
  \end{array}
\right),
\end{align*}
which is equivalent to
\begin{align*}
\left(
  \begin{array}{cc}
   i&  i\eta \\ [5pt]
   0 &  -2k\eta \\
  \end{array}
\right)
\left(
  \begin{array}{c}
  e^{ig} \\ [5pt]e^{-ig}
  \end{array}
\right)
=\left(
  \begin{array}{cc}
   k_* &  -k_*\eta\\ [5pt]
   (k_*k+1)i &  -(k_*k-1)\eta i \\
  \end{array}
\right) \left(
  \begin{array}{c}
  e^{ig_*} \\ [5pt]e^{-ig_*}
  \end{array}
\right).
\end{align*}
If follows from either $k_*k=1$ or $k_*k=-1$ that either
$$
-2k\eta e^{-ig}=2ie^{ig_*} \quad \mbox{or} \quad -2ke^{-ig} =2ie^{-ig_*},
$$
respectively. The former shows that $g+g_*$ is constant unless $\eta$ is constant and the latter shows that $g-g_*$ is constant.
\end{proof}

\begin{remark} If vectors $(\eta_1, \eta_2, \eta_5)$ and $(\eta_3, \eta_4, \eta_6)$ are linearly independent, then one of
$$
\left(
  \begin{array}{cc}
   \eta_1 &  \eta_2  \\ [5pt]
   \eta_3 &  \eta_4 \\
  \end{array}
\right), \quad
\left(
  \begin{array}{cc}
   \eta_1 &  \eta_5  \\ [5pt]
   \eta_3 &  \eta_6 \\
  \end{array}
\right),
\quad \mbox{and} \quad
\left(
  \begin{array}{cc}
   \eta_2 &  \eta_5  \\ [5pt]
   \eta_4 &  \eta_6 \\
  \end{array}
\right)
$$
must be invertible. Thus, we can adjust $\eta_1, \eta_2 \cdots, \eta_6$ accordingly to make the lemma be true.
\end{remark}

\begin{remark}\label{rem2.3}
Lemma \ref{equiv} tells us that a solution $f$ to \eqref{lem1eq} may depend on $g$ and further, if there are two such $g$'s, say, $g$ and $g_*$, then either $g+g_*$ or $g-g_*$ is constant. This fact could lead to two different representations of $f$ in theorems and examples in later sections. For simplicity, we only give one representation of $f$ with a fixed $g$, with which the other representation can be easily obtained if we replace $g$ by $g_*$. We provide an example in Remark \ref{rem3.9}.
\end{remark}

\begin{lemma}\cite[Theorem 4.1]{chang-li-yang-1995} \label{L1} If $f$ is a transcendental meromorphic function in $\mathbb{C}$ and $g$ is a transcendental entire function on $\mathbb{C}^{n},$ then $$\lim_{r\rightarrow\infty}\frac{T(r, f(g))}{T(r, g)}=+\infty.$$
\end{lemma}

\begin{lemma} \label{new1}
If $g$ is a non-constant entire function on $\mathbb{C}^{n},$ then
$$\lim_{r\rightarrow\infty}\frac{T(r, e^g)}{T(r, g)}=+\infty.$$
\end{lemma}
The above lemma can be proved straightforwardly by considering whether $g$ is a polynomial or a transcendental function along with Lemma \ref{L1}.

\begin{lemma}\cite[Theorem 1.106]{hu-li-yang-book-2003}\label{borel}
Suppose that $a_0, a_1, \cdots, a_m (m\ge 1)$ are meromorphic functions on $\C^n$ and $g_0, g_1, \cdots, g_m$ are entire functions on $\C^n$ such that $g_j-g_k$ are not constant for $0\le j<k\le m$.
If
$$
\sum_{j=0}^ma_j(z)e^{g_j(z)}\equiv 0
$$
and $T(r, a_t)=o\left( \min_{0\le j<k\le m}\{T(r, e^{g_j-g_k})\}\right)$, for $t=0, 1, \cdots, m$ and for all $r$ except possibly a set of finite Lebesgue measure, then $a_t\equiv 0$ for
$t=0, 1, \cdots, m$.
\end{lemma}

\begin{lemma} \label{L5} Let $f$ be a non-constant meromorphic function on $\mathbb{C}^{n}.$ Then for any $I\in (\mathbb{Z}^{+})^{n},$ we have
$$T(r, \partial^{I}f )=O(T(r, f))$$ for all $r$ except possibly a set of finite Lebesgue measure.
\end{lemma}
The proof of the above lemma is an application of the logarithmic derivative lemma.

\begin{lemma}\label{const}
Let $h$ be a polynomial on $\C^n$. If $h+\underline{h}=0$ or $h+\uh$ is a non-zero constant,
then $h\equiv 0$ or $h$ is a non-zero constant, respectively.
\end{lemma}
\begin{proof} Set  $c=(c_1, c_2,\cdots, c_n).$
Assume that $h_*(z):=c_*\prod_{j=1}^nz_j^{k_j}$ is a highest degree term of $h$ with $c_*\neq 0$.
(Note $h$ may have many highest degree terms.) Thus,
\begin{align*}
h_*(z)+h_*(z-c) & = c_*\left(\prod_{j=1}^nz_j^{k_j} + \prod_{j=1}^n(z_j-c_j)^{k_j}\right) \\
& =2c_*\prod_{j=1}^nz_j^{k_j}+ \{\mbox{terms with degree $<k_1+\cdots+k_n$}\}.
\end{align*}
Since $h(z)+h(z-c)$ is constant, so $c_*=0$ or $h_*$ is constant.
\end{proof}

\begin{lemma}\cite[Lemma 3.1]{hu-li-yang-book-2003}\label{L7} Let $f_{j}(\not\equiv 0),$ $j=1, 2, 3,$ be meromorphic functions on $\mathbb{C}^{n}$ such that $f_{1}$ is not constant, $f_{1}+f_{2}+f_{3}=1,$ and
 $$\sum_{k=1}^{3}\left\{N^{2}(r,\frac{1}{f_{k}})+2\overline{N}(r, f_{k})\right\}\le \lambda T(r, f_{1})+O(\log^{+}T(r, f_{1}))$$
holds for all $r$ outside possibly a set with finite logarithmic measure, where $\lambda<1$ is a positive number. Then $f_{2}=1$ or $f_{3}=1.$
\end{lemma}

The following lemmas play a key role in the proofs of our theorems.\par

\begin{lemma}\label{new3}
Let $g$ and $u$ be entire functions on $\C^n$ satisfying the equation
\begin{equation}\label{new3-01}
\alpha e^{ig}+\beta e^{-ig}+\gamma e^{i\ug}+\delta e^{-i\ug}=0,
\end{equation}
where $\alpha$, $\beta$, $\gamma$ and $\delta$ are polynomials in $u$ with rational coefficients.
If  exactly one of $\alpha, \beta, \gamma$ and $\delta$ is identically equal to $0$ and $T(r, u)=O(T(r, \ug))$, then $g$ is constant.
\end{lemma}
\begin{proof}
First, we consider the situation that $\alpha \equiv 0$ and $\beta\gamma\delta\not\equiv 0$. Thus,  \eqref{new3-01} can be written as
\begin{equation*}
-\frac{\gamma}{\beta} e^{i(g+\underline{g})}-\frac{\delta}{\beta} e^{i(g-\underline{g})}=1.
\end{equation*}
Set $F:=-\frac{\gamma}{\beta} e^{i(g+\underline{g})}, G:=-\frac{\delta}{\beta}e^{i(g-\underline{g})}.$ Then $F(z)+G(z)\equiv 1$ and $T(r, F)=T(r,G)+O(1)$.
Applying the second main theorem to $F$ yields
 \begin{eqnarray*}
T(r, F)&\leq&N(r, F)+N(r, \frac{1}{F})+N(r, \frac{1}{F-1})+o(T(r, F))\\
&=&N(r, F)+N(r, \frac{1}{F})+N(r, \frac{1}{G})+o(T(r, F))\\
&\leq&N(r, \frac{1}{\beta})+N(r, \frac{1}{\gamma})+N(r, \frac{1}{\delta})+o(T(r, F))\\
&\leq&O(T(r,\underline{g}))+o(T(r, F))+O(\log r),
\end{eqnarray*}
for all large $r$ outside a set of finite Lebesgue measure. Consequently,
$$
T(r, F)\le O(T(r,\underline{g}))+O(\log r)
$$
for all large $r$ outside a set of finite Lebesgue measure.
Similarly, we also have
$$T(r,G)\leq O(T(r,\underline{g}))+O(\log r)$$
for all large $r$ outside a set of finite Lebesgue measure.
Applying the above two inequalities, along with the first main theorem, to  $$e^{2i\underline{g}}=\frac{\delta}{\gamma}\Big(\frac{F}{G}\Big)$$
gives
\begin{equation}\label{new03-2}
T(r,e^{2i\underline{g}})\leq O(T(r,\underline{g}))+O(\log r)
\end{equation}
for all large $r$ outside a set of finite Lebesgue measure.
By Lemma \ref{new1}, $g$ is constant.

Secondly, we consider the situation that $\beta\equiv 0$ and $\alpha\gamma\delta\not\equiv 0$. Thus,  \eqref{new3-01} can be written as
\begin{equation*}
-\frac{\gamma}{\alpha} e^{-i(g-\underline{g})}-\frac{\delta}{\alpha} e^{-i(g+\underline{g})}=1.
\end{equation*}
Set $F\stackrel{def}{=}-\frac{\gamma}{\alpha} e^{-i(g-\underline{g})}, G\stackrel{def}{=}-\frac{\delta}{\alpha}e^{-i(g+\underline{g})}.$ Then
$$
F(z)+G(z)\equiv 1 \quad \mbox{and} \quad e^{2i\underline{g}}=\frac{\delta}{\gamma}\Big(\frac{F}{G}\Big).
$$
Similar to the first situation, we get $g$ is constant.

Thirdly, we consider the situation that  $\gamma\equiv 0$ and $\alpha\beta\delta\not\equiv 0$. Then, \eqref{new3-01} can be written as
\begin{equation*}
-\frac{\beta}{\alpha} e^{-2ig}-\frac{\delta}{\alpha} e^{-i(g+\underline{g})}=1.
\end{equation*}
Set $F\stackrel{def}{=}-\frac{\beta}{\alpha} e^{-2ig}, G\stackrel{def}{=}-\frac{\delta}{\alpha} e^{-i(g+\underline{g})}.$ Then $F(z)+G(z)\equiv 1$.
Similarly, applying the second main theorem to $F$, as well as $G$, yields
 $$
T(r, F)\leq O(T(r,\underline{g}))+O(\log r)
$$
and
$$
T(r, G)\leq O(T(r,\underline{g}))+O(\log r)
$$
for all large $r$ outside a set of finite Lebesgue measure.
Taking the two inequalities into
 $$e^{2i\underline{g}}=-\frac{\delta^2}{\alpha\beta}\Big(\frac{F}{G^2}\Big)$$
gives \eqref{new03-2}, which shows that $g$ is constant.

Lastly, we can treat the situation that
$\delta \equiv 0$  and $\alpha\beta\gamma\not\equiv 0$ as we do in the third situation to obtain $g$ is constant.
\end{proof}

\begin{lemma}\label{new2}
Let $g$ and $u$ be entire functions on $\C^n$ satisfying the equation
\eqref{new3-01}, where $g$ is non-zero and
$$
\alpha\beta-\gamma\delta
=s_mu^m+s_{m-1}u^{m-1}+\cdots+s_0,
$$
where $s_m\not\equiv 0$ and $s_j$ ($j=0,1, \cdots, m$) are rational functions.
If $\alpha\beta\gamma\delta\not\equiv 0$ and $T(r, u)=O(T(r, \ug))$, then $\alpha\beta-\gamma\delta\equiv 0$ and either
$$
e^{i(g+\ug)}=-\frac{\delta}{\alpha} \quad \mbox{or} \quad e^{-i(g-\ug)}=-\frac{\delta}{\beta}.
$$
In addition, if $m\ge 1$, then  $u$ is a polynomial and either $g+ \ug$ and ${\delta}/{\alpha}$ are constant or  $g-\ug$ and ${\delta}/{\beta}$ are constant.
\end{lemma}
\begin{proof}
We rewrite \eqref{new3-01} as
\begin{equation}\label{new1-02}
-\frac{\gamma}{\delta}e^{2i\underline{g}}-\frac{\alpha}{\delta}e^{i(g+\underline{g})}-\frac{\beta}{\delta}e^{-i(g-\underline{g})}=1.
\end{equation}
Set
$$
f_1=-\frac{\gamma}{\delta}e^{2i\underline{g}},\,\quad f_2=-\frac{\alpha}{\delta}e^{i(g+\underline{g})}\quad \mbox{and} \ \quad f_3=-\frac{\beta}{\delta}e^{-i(g-\underline{g})}.
$$
Then, by the first main theorem and  Lemma \ref{new1}, we obtain
\begin{equation*}
\begin{split}
\max_{j=1,2,3}\big\{ N(r, f_j), N\left(r, \frac{1}{f_j}\right)\big\}
& =O(T(r,\alpha)+T(r,\beta)+T(r,\gamma)+T(r,\delta))\\
&= O(T(r, u ))+ O(\log r)\\
& =O(T(r,\underline{g}))+O(\log r)
=o(T(r,f_1)).
\end{split}
\end{equation*}
Applying Lemma \ref{L7} to \eqref{new1-02} yields $f_2=1$ or $f_3=1$. When $f_2=1$, we have from \eqref{new1-02} that $f_1+f_3=0$. Hence,
$$
\quad\left(-\frac{\gamma}{\delta}+\frac{\alpha\beta}{\delta^2}\right)e^{2i\underline{g}}=0,$$
which implies $\alpha\beta-\gamma\delta\equiv 0$.
When $f_3=1$, similarly, we have
$$\frac{\gamma}{\delta}e^{2i\underline{g}}=-\frac{\alpha}{\delta}e^{i(g+\underline{g})}= \frac{\alpha\beta}{\delta^2}e^{2i\underline{g}}\qquad \text{and}\qquad \alpha\beta-\gamma\delta\equiv 0.
$$
Therefore, $\alpha\beta-\gamma\delta\equiv 0$ regardless $f_2 \equiv 1$ or $f_3\equiv 1$. The first part of the lemma is proved.

If $m\ge 1$, then, by the definition of
$
\alpha\beta-\gamma\delta
=s_mu^m+s_{m-1}u^{m-1}+\cdots+s_0,
$
we obtain
\begin{equation*}
-u=\frac{s_{m-1}}{s_m}+\frac{s_{m-2}}{s_m}\frac{1}{u }+\frac{s_{m-3}}{s_m}\frac{1}{u^2}+\cdots+ \frac{s_{0}}{s_m}\frac{1}{u^{m-1} }.
\end{equation*}
 It follows that
\begin{align*}
T(r, u)
&=
\int_{S_{n}(r)}\log^{+}|u(z)|\sigma_{n}(z)=\int_{S_{n}(r)\cap \{z: |u(z)|\ge 1\}}\log^{+}|u(z)|\sigma_{n}(z)\\
&\le
\int_{S_{n}(r)\cap \{z: |u(z)|\ge 1\}}\log^{+}\left||\frac{s_{m-1}}{s_m}|+|\frac{s_{m-2}}{s_m}|+\cdots+ |\frac{s_{0}}{s_m}|\right| \sigma_{n}(z)\\
&= O(\log r).
\end{align*}
Thus, $u$ is a polynomial. Consequently, $\alpha, \beta, \gamma$ and $\delta$ are rational functions.
Furthermore, $f_2\equiv1$ and $f_3\equiv 1$ are equivalent to
$$
e^{i(g+\ug)}=-\frac{\delta}{\alpha} \quad \mbox{and} \quad e^{-i(g-\ug)}=-\frac{\delta}{\beta}
$$
which deduce that $g+\og$ and $g-\ug$ are constant, respectively.
\end{proof}

\section{Case I. $D\not\equiv 0$ and $d_2\equiv 0$}

\begin{theorem} \label{T1}
Let $p,p_1,p_2,\cdots,p_6$ be polynomials, where $p$ is a non-zero irreducible polynomial. If $D\not\equiv 0$ and $d_2\equiv 0$, then every entire solution to \eqref{E1.1}
on $\mathbb{C}^{n}$ has the form:
\begin{equation}\label{1-14-3}
f(z)
= \frac{\underline{a_1}e^{i\underline{g}}-\underline{pa_2}e^{-i\underline{g}}}{2i\underline{D}},
\end{equation}
where $g$ is an entire function having one of the following properties:
\begin{enumerate}
\item[(i)] $g$ is constant.
\item[(ii)] $L(\ug)$ is a polynomial and either $g+\underline{g}$ or $g-\underline{g}$ is constant.
\item[(iii)] $L(\ug)$ is transcendental when $g+\ug$ is a non-constant entire function and
either $a_1\equiv b_2\equiv 0$  or $a_2\equiv b_1\equiv 0$.
\end{enumerate}
In addition, if $L(f)=f_{ z_{j}}$ and $d_1\equiv d_2\equiv 0$ for any $j\in\{1,2,\cdots,n\}$,
then every entire solution $f$ has the form as in (\ref{1-14-3}), where
$g$ is an entire function with one of the following properties:
\begin{enumerate}
\item[(a)] $g$ is constant.
\item[(b)] $g_{z_j}\equiv 0$ and either $g+\underline{g}$ or $g-\underline{g}$ is constant.
\item[(c)] $g_{z_j}$ is a non-zero polynomial function, $g-\underline{g}$ is constant, and
either
$\deg_{z_j}\underline{g}_{z_j}=\deg_{z_j}b_1 -\deg_{z_j}a_1 $ or
$\deg_{z_j}\underline{g}_{z_j}=\deg_{z_j}b_2 -\deg_{z_j}a_2$.
\item[(d)] $g_{z_j}$ is transcendental when $g+\ug$ is a non-constant entire function and
either $a_1\equiv b_2\equiv 0$  or $a_2\equiv b_1\equiv 0$.
\end{enumerate}
\end{theorem}

\begin{proof}
Let $f$ be an entire solution to \eqref{E1.1}. Applying Lemma \ref{equiv} with
$$s(f)=L(f), \quad t(f)=\overline{f}\quad \mbox{and} \quad u(f)=f,$$
we obtain
\begin{align}\label{E1.2}
\left(\begin{array}{c}
    L(f)  \\ [5pt] \overline{f}
\end{array}\right)
= \frac{1}{2iD}\left(
  \begin{array}{cc}
    -{b_1} & p{b_2} \\ [5pt] a_1 & -pa_2
  \end{array}
\right)\left(
  \begin{array}{c}
   e^{ig} \\ [5pt]  e^{-ig}
  \end{array}
\right)+
 \frac{1}{D} \left(
  \begin{array}{c}
  d_1 \\ [5pt] 0
  \end{array}
\right) f.
\end{align}
Recall that this matrix equation is equivalent to \eqref{E1.1} by Lemma \ref{equiv}.
It follows that
$$
{f} =
\frac{1}{2i\underline{D}}\left(
 \underline{a_1} \  -\underline{pa_2}
    \right)
    \left(
  \begin{array}{c}
   e^{i\underline{g}} \\  e^{-i\underline{g}}
  \end{array}
\right)\stackrel{def}{=}f_1 \underline{E},
$$
where $f_1=\frac{1}{2i\underline{D}}\left(
 \underline{a_1} \  -\underline{pa_2}
    \right)$,
${E}=( e^{i{g}} \quad e^{-i{g}})^{T}$
     and
$\underline{E}=( e^{i\underline{g}} \quad e^{-i\underline{g}})^{T}$ is a $2$ by $1$ matrix (the matrix $C^T$ is the transpose of the matrix $C$).
Thus, the representation \eqref{1-14-3} is proved.

Noting that
\begin{equation}\label{1-14-3.1}
L(\underline{E})=iL(\ug)\left(
  \begin{array}{cc}
    1 & 0 \\ [5pt] 0 & -1
  \end{array}
\right) \underline{E},
\end{equation}
we obtain
\begin{equation}\label{1-14-3.2}
L(f)=\left[L(f_1)+iL(\ug) f_1\left(
  \begin{array}{cc}
    1 & 0 \\ [5pt] 0 & -1
  \end{array}
\right)\right] \underline{E}.
\end{equation}

On the other hand, \eqref{E1.2} shows
\begin{equation}\label{yenew27}
L(f)=\frac{1}{2iD}\left(
   -b_1\quad  pb_2
  \right)E+\frac{d_1}{D}f_1\underline{E}\stackrel{def}{=}f_2 {E}+f_3f_1 \underline{E}.
\end{equation}
It follows from \eqref{1-14-3.2} and \eqref{yenew27} that
$$
f_2E+f_3f_1 \underline{E}=\left[L(f_1)+iL(\ug) f_1\left(
  \begin{array}{cc}
    1 & 0 \\ [5pt] 0 & -1
  \end{array}
\right)\right] \underline{E},
$$
where $f_1, f_2$ and $f_3$ are independent of $g$.

The above identity can be written as
\begin{equation}\label{E1.7}
\alpha e^{ig}+\beta e^{-ig}=\gamma e^{i\underline{g}}+\delta e^{-i\underline{g}},
\end{equation}
where
\begin{eqnarray*}
\alpha(z)& = & -\underline{D}^{2}b_1, 
\qquad  \qquad \beta(z) =\underline{D}^{2}pb_2, \\ 
\gamma(z)
&=& D\left[ -\underline{a_1}L(\underline{D})+\underline{D}(L(\underline{a_1})+i\underline{a_1}L(\underline{g}))\right]-d_1\underline{a_1}\underline{D},\\
\delta(z)
& =& D\left[\,\underline{pa_2}L(\underline{D})+\underline{D}(-L(\underline{pa_2})+
   i\underline{pa_2}L(\underline{g}))\right]+d_1\underline{pa_2}\underline{D}.
\end{eqnarray*}
\par

The assumption $D\not\equiv 0$ implies that $\alpha$ and $\beta$ cannot be identical to zero at the same time. If there exist three of $\alpha, \beta, \gamma$ and $\delta$ are identically equal to zero, then all $\alpha, \beta, \gamma$ and $\delta$ are identically equal to zero. This is impossible. Below we prove the theorem by considering three cases.\vskip 1mm
\par
{\it Case 1:} exactly two of  $\alpha, \beta, \gamma$ and $\delta$ are identically equal to zero.
\vskip 1mm
\par {\it Sub-case 1.1:}\,
$\alpha\equiv \gamma\equiv 0$  and $\beta\delta\not\equiv 0$. Clearly, $\alpha\equiv 0$ implies $ b_1\equiv 0$, and further, $ a_1\not\equiv 0$ since $D=ka_1b_2/(2i)\not\equiv 0$. Then  $\gamma \equiv 0$ leads
\begin{equation*}
L(\underline{g})=i\frac{L(\underline{a_1})}{\underline{a_1}}-i\frac{L(\underline{D})}{\underline{D}}-i\frac{d_1}{D}=-i\frac{L(\underline{b_2})}{\underline{b_2}}-i\frac{d_1}{D}
\end{equation*}
which means that $L(\underline{g})$ is a polynomial. So is $\delta$. It follows from \eqref{E1.7} that
$$e^{i(g-\underline{g})}=\frac{\beta}{\delta},$$
then $g-\underline{g}$ is constant. Hence, $g$ has the property (ii) in Theorem \ref{T1}.\vskip 1mm
\par {\it Sub-case 1.2:}\, $\alpha\equiv \delta\equiv 0$ and $\beta\gamma\not\equiv 0$. Noting $\alpha\equiv 0$ if and only if $ b_1\equiv 0$, we further break it into following two situations.\vskip 1mm
\par
The first situation: $\alpha\equiv \delta\equiv b_1\equiv a_2\equiv 0$ and $\beta\gamma\not\equiv 0$. Thus,
$a_1=2k p_1$, $b_2=2kp_2$ and $D=-2kip_1p_2$.
We obtain from \eqref{E1.2} that
\begin{equation*}
\left(\begin{array}{c}
  L(f) \\ [5pt]  \overline{f}
\end{array}\right)
= \left(
  \begin{array}{cc}
   0 &  \frac{p}{2p_1} \\ [5pt] \frac{1}{2p_2}  & 0
  \end{array}
\right)\left(
  \begin{array}{c}
   e^{ig} \\ [5pt]  e^{-ig}
  \end{array}
\right)+
 \frac{1}{D} \left(
  \begin{array}{c}
  d_1 \\ [5pt] 0
  \end{array}
\right) f.
\end{equation*}
Since $\overline{f}$ and $e^{ig}$ are non-zero entire functions, $1/(2p_2)$ is  a constant function, say, $c_*\neq 0$.
Therefore, $\overline{f}= c_*e^{ig}$ implies that
$$f=c_*e^{i\underline{g}} \qquad \mbox{and} \qquad L(f)=c_*i e^{i\underline{g}}L(\underline{g}).$$
It follows that
\begin{equation}\label{yeAdd3}
c_*iL(\ug)=\frac{p}{2p_1}e^{-i(g+\underline{g})}+\frac{d_1}{D}c_*.
\end{equation}
Thus,  $L(\underline{g})$ is a polynomial if and only if $g+\underline{g}$ is constant,  which means $g$ has the property (ii) in Theorem \ref{T1}; and  $L(\underline{g})$ is a transcendental entire function if and only if  $g+\underline{g}$ is a non-constant entire function, which means $g$ has the property (iii) in Theorem \ref{T1}

The second situation: $ \alpha\equiv \delta\equiv b_1\equiv 0$ and $a_2\beta\gamma\not\equiv 0$. Thus, $\delta \equiv 0$ and \eqref{E1.7} respectively yield
\begin{eqnarray*}
L(\underline{g})=i\frac{L(\underline{D})}{\underline{D}}-i\frac{L(\underline{pa_2})}{\underline{pa_2}}+i\frac{d_1}{D}\quad \text{and}\quad e^{i(g+\underline{g})}=\frac{\beta}{\gamma},
\end{eqnarray*}
which implies that $L(g)$ is a polynomial and $g+\underline{g}$ is constant as in the proof of {\it Sub-case 1.1}. Consequently,
$g$ has the property (ii) in Theorem \ref{T1}.\vskip 1mm
\par {\it Sub-case 1.3:} $\beta\equiv \gamma \equiv 0$ and $\alpha\delta\not\equiv0$. Then $b_2 \equiv 0$ and $b_1=2kp_2\not\equiv 0$.
If, in addition, $a_1\not\equiv 0$, it follows from $\gamma\equiv 0$ and \eqref{E1.7} that
\begin{equation*}
L(\underline{g})=i\frac{L(\underline{a_1})}{\underline{a_1}}-i\frac{L(\underline{D})}{\underline{D}}-i\frac{d_1}{D}\,\,\,\text{and}\,\,\,e^{-i(g+\ug)}=\frac{\alpha}{\delta}.
\end{equation*}
This implies that $L(\ug)$ is a polynomial, and
 $g+\underline{g}$ is constant, so $g$ has the property (ii) in Theorem \ref{T1}.

If  $ a_1\equiv 0$, then $a_2=2k p_1\not\equiv 0$ and
$D\equiv 2ki p_1p_2$. It turns out  from \eqref{E1.2} that
\begin{equation*}
\left(\begin{array}{c}
   L(f) \\ [5pt] \overline{f}
\end{array}\right)
= \left(
  \begin{array}{cc}
 \frac1{2p_1} & 0 \\ [5pt] 0& \frac{p}{2p_2}
  \end{array}
\right)\left(
  \begin{array}{c}
   e^{ig} \\  [5pt] e^{-ig}
  \end{array}
\right)+
 \frac{1}{D} \left(
  \begin{array}{c}
  d_1 \\ [5pt] 0
  \end{array}
\right) f.
\end{equation*}
Since $\overline{f}$ and $e^{-ig}$ are entire functions, so, $p/(2p_2)$ is a non-zero polynomial, say,
$p_*$.
Also, $\overline{f}=\frac{p}{2p_2} e^{-ig}=p_*e^{-ig}$, which deduces
$$
f=\underline{p_*}e^{-i\ug} \qquad \mbox{and} \qquad L(f)=L(\underline{p_*})e^{-i\ug}-\underline{p_*} e^{-i\ug} iL(\ug).
$$
It follows that
\begin{equation}\label{wang-14-n}
\frac{1}{2p_1}e^{i(g+\ug)}+\frac{d_1}{D}\underline{p_*}=L(\underline{p_*})-i\underline{p_*}L(\ug).
\end{equation}
Therefore,
$L(\underline{g})$ is a polynomial if and only if  $g+\underline{g}$ is constant, which shows that
$g$ has the property (ii) in Theorem \ref{T1}.
Also,  $L(\underline{g})$ is transcendental if and only if $g+\underline{g}$ is a non-constant entire function.
Thus, $g$ has the property (iii) in Theorem \ref{T1}.\vskip 1mm
\par
{\it Sub-case 1.4:}\, $\beta\equiv \delta \equiv 0$ and $\alpha\gamma\not\equiv0$. Cleary, $ b_2\equiv 0$, and $ a_2 \not\equiv 0$ due to $D\not\equiv 0$. From $\delta \equiv 0$
and \eqref{E1.7}, we also have
$$
L(\underline{g})=i\frac{L(\underline{D})}{ \underline{D}}-i
\frac{L(\underline{pa_2})}{\underline{pa_2}}+i\frac{d_1}{D}\,\,\,\text{and}\,\,\,e^{-i(g-\underline{g})}=\frac{\alpha}{\gamma}.
$$
This means that $L(\underline{g})$ is a polynomial and $g-\underline{g}$ is constant. Therefore, $g$ has the property (ii) in Theorem \ref{T1}.\vskip 1mm
\par {\it Sub-case 1.5:}\, $\gamma\equiv \delta \equiv 0$ and $\alpha\beta\not\equiv 0$. Then \eqref{E1.7} gives
$$e^{-2ig}=-\frac{\alpha}{\beta}.$$
It follows from $\alpha/\beta$ is a polynomial that $g$ is constant. Therefore, $g$ has the property
(i) in Theorem \ref{T1}.\vskip 1mm

\par
{\it Case 2:} exactly one of  $\alpha, \beta, \gamma$ and $\delta$ are identically equal to zero. This is a straightforward
application of Lemma \ref{new3} with $u=L(\ug)$. Thus, $g$ has the property (i) in Theorem \ref{T1}.
\vskip 1mm
\par {\it Case 3:}\, $\alpha\beta\gamma\delta\not\equiv 0$.
If $L(\ug)$ is constant, then $\alpha, \beta, \gamma$ and $\delta$ are polynomial. Thus, the first part of Lemma \ref{new2} shows that $g+\ug$ or $g-\ug$ is constant. If $L(\ug)$ is non-constant,  then the second part of Lemma \ref{new2} with $u=L(\ug)$ and $m=2$ derives
that $L(\ug)$ is a polynomial and $g+\ug$ or $g-\ug$ is constant.
Therefore, in this sub-case, $g$ has the property (ii) in Theorem \ref{T1}. The first part of the theorem is proved.\vskip 2mm
\par When $L(f)=f_{z_j}$ and $d_1\equiv 0$, we only need further discuss the above cases and sub-cases which yields the property (ii).\vskip 1mm
\par {\it Case A:}\,exactly two of $\alpha,\beta,\gamma,\delta$ are identical to zero.\vskip 1mm
\par {\it Sub-case A.1:}\, $\alpha\equiv\gamma\equiv 0$ and $\beta\delta\not\equiv 0$. Since $a_1\not\equiv 0$, $\gamma\equiv 0$ deduces that
\begin{equation}\label{24-4-17}
\underline{g}_{z_{j}}=i\frac{(\underline{a_1})_{z_j}}{\underline{a_1}}-i\frac{\underline{D}_{z_j}}{\underline{D}}
\end{equation}
is a polynomial in $z$. When we fix the $n-1$ variables $z_i(i\not=j)$, $\underline{g}_{z_{j}}$ becomes a polynomial in $z_j$ satisfying $\underline{g}_{z_{j}}\to 0$ as $z_j\to\infty$. So, $\underline{g}_{z_{j}}\equiv 0$ by Liouville's theorem. From \eqref{E1.7}, $g-\ug$ is constant.
Hence, $g$ has the property (b) in Theorem \ref{T1}.\vskip 1mm
\par {\it Sub-case A.2:}\,$\alpha\equiv\delta\equiv 0$ and $\beta\gamma\not\equiv0$.
In the situation $a_2\not\equiv 0$, $\delta\equiv 0$ leads
\begin{eqnarray}\label{24-4-171}
(\underline{g})_{z_j}=i\frac{(\underline{D})_{z_j}}{\underline{D}}-i\frac{(\underline{pa_2})_{z_j}}{\underline{pa_2}},
\end{eqnarray}
which implies that $(g)_{z_j}\equiv 0$ by Liouville's theorem.  So, $\gamma$ is a polynomial. Noting
$e^{-i(g-\underline{g})}={\beta}/{\gamma}$,  we have the property (b) in Theorem \ref{T1}.\vskip 1mm
\par {\it Sub-case A.3:} $\beta \equiv \gamma \equiv 0$ and $\alpha\delta\not\equiv 0$. If $a_1\not\equiv 0$, $\gamma\equiv 0$ deduces \eqref{24-4-17} again, which is a polynomial. Then $\underline{g}_{z_{j}}\equiv 0$ by Liouville's theorem. From \eqref{E1.7}, $g+\underline{g}$ must be constant, so $g$ has the property (b) in Theorem \ref{T1}.
\par
If $a_1\equiv 0$, it follows from \eqref{wang-14-n} that we get
$$
c_*e^{i(g+\ug)}=(\underline{p_*})_{z_j}-i\underline{p_*}(\ug)_{z_j}.
$$
Then $(\underline{g})_{z_j}$ is a polynomial if and only if $g+\underline{g}$ is constant, which implies
that $g_{z_j}+\underline{g_{z_j}}=0$ since $\underline{g_{z_j}}=(\underline{g})_{z_j}$.  By Lemma \ref{const}, $g_{z_j}\equiv 0$.  Thus,
$g$ has the property (b) again. \vskip 1mm
\par  {\it Sub-case A.4:}\,$\beta\equiv \delta\equiv 0$ and $\alpha\gamma\not\equiv0$. Since $a_2\not\equiv 0$, $\delta\equiv 0$ implies that $(\underline{g})_{z_j}$ is given in \eqref{24-4-171}.
By Liouville's theorem, $(\underline{g})_{z_j}\equiv 0$, and again from \eqref{E1.7},
$g-\underline{g}$ is constant. So, $g$ has the property (b) in Theorem \ref{T1}.\vskip 1mm

\par {\it Case B:}\,$\alpha\beta\gamma\delta\not\equiv 0$.
Applying Lemma \ref{new2} with $u=(\ug)_{z_j}$ to \eqref{E1.7}, we obtain
that $\alpha\beta-\gamma\delta\equiv 0$, $g_{z_j}$ is a polynomial and either $g+\ug$ and $e^{i(g+\ug)}=\delta/\alpha$ are constants
or $g-\ug$ and $e^{i(g-\ug)}=\beta/\delta$ are constants.\vskip 1mm
\par
We claim that $g_{z_{j}}\equiv 0$ when $g+\ug$ and $\delta/\alpha$ are constants. Since $g+\ug$ is constant, $g_{z_j}+\underline{g}_{z_j}=0$. By Lemma \ref{const}, $g_{z_j}\equiv 0$.
Therefore, in this situation, $g$ has the property (b) in Theorem \ref{T1}.

Now, we consider the situation when $g-\ug$ and $\beta/\delta=\gamma/\alpha$ are constants.
Since $D=k(a_1b_2-a_2b_1)/(2i)\not\equiv 0$, so, either $a_1\not\equiv 0$ or $a_2\not\equiv 0$.

Suppose $a_1\not\equiv 0$. Since $\gamma/ \alpha$ is a nonzero constant, we write it as $c_*$. It follows from the definitions of $\gamma$ and $\alpha$ that
\begin{equation}\label{f2-0}
\underline{g}_{z_j}= -c_*\frac{\underline{D}b_1}{iD\underline{a_1}} -\frac{(\underline{a_1})_{z_j}}{i\underline{a_1}}+\frac{\underline{D}_{z_{j}}}{i\underline{D}}.
\end{equation}

If $\deg_{z_j}a_1> \deg_{z_j}b_1$, then \eqref{f2-0} deduces that $g_{z_j}$ goes to zero as $z_j\to \infty$. By Liouville's theorem,
$g_{z_j}\equiv 0$.  Thus, in this situation, $g$ has the property (b) in Theorem \ref{T1}.
If $\deg_{z_j}a_1\le \deg_{z_j}b_1$, then  \eqref{f2-0} shows that $g_{z_j}$ is a polynomial and
$\deg_{z_j}g_{z_j}= \deg_{z_j}b_1-\deg_{z_j}a_1\ge 0$. Thus,
 in this situation, $g$ has the property (c) in Theorem \ref{T1}.

 Suppose $a_2\not\equiv 0$. It follows from
 $\delta/\beta=\alpha/\delta=1/c_*$ that
$$
\underline{g}_{z_j} = -\frac{1}{c_*}\frac{\underline{D}pb_2}{iD\underline{pa_2}} -\frac{(\underline{pa_2})_{z_j}}{i\underline{pa_2}}+\frac{\underline{D}_{z_{j}}}{i\underline{D}}
$$
converges to zero if $\deg_{z_j}a_2> \deg_{z_j}b_2$. Consequently, $g_{z_j}\equiv 0$. Thus, in this situation, $g$ has the property (b) again.
If $\deg_{z_j}a_2 \le  \deg_{z_j}b_2$, then
$g_{z_j}$ is a polynomial, $\deg_{z_j}g_{z_j}= \deg_{z_j}b_2-\deg_{z_j}a_2\ge 0$, and
$g$ has the property (c) in theorem \ref{T1}.

\end{proof}

\begin{remark}
The entire solution $f$ in Theorem \ref{T1} can also be written as
$$
f(z)
=\frac{1}{\underline{D}}\left(k\underline{p_{1}}\frac{e^{i\underline{g}}-\underline{p}e^{-i\underline{g}}}{2i}-
\underline{p_{3}}\frac{e^{i\underline{g}}+\underline{p}e^{-i\underline{g}}}{2}\right).
$$
The function $g$ here is the same as in Theorem \ref{T1}.
\end{remark}
\begin{remark}
When $n=1$, $g$ is a polynomial of one variable  and $g(z)-\underline{g}(z)$ is constant in $\mathbb{C}$, then $g$ is a linear function. Further, when $n=1$, that $L(g)$ is a polynomial implies that $g$ is polynomial. Thus,
the property (ii) in Theorem \ref{T1} can be replaced by $g$ is a polynomial of degree one. Hence, when $n=1$, $g$ is either a linear or a transcendental function.
\end{remark}
\begin{remark}
When the conditions in the property (iii) in Theorem \ref{T1} are satisfied, we have from the proof of Theorem \ref{T1} that either
$\displaystyle f(z)=c_*e^{i\ug}$ or $\displaystyle f(z)=p_*e^{-i\ug}$, where $c_*$ is a non-zero complex number and $p_*$ is a non-zero polynomial.
\end{remark}
\begin{remark}
If one of $p_1,p_2,p_3,p_4$ is identically equals to zero, then the property (iii) in Theorem \ref{T1} does not occur, i.e., $|a_1|+|b_2|\not\equiv 0$ or $|a_2|+|b_1|\not\equiv 0$ due to the fact that $D\not\equiv 0$.
\end{remark}
\vskip.05in
\par
When $p_2=p_3=p_6=0$ and $p_1=p_4=p=1$ in \eqref{E1.1}, we get a necessary-sufficient theorem as below.

\begin{corollary}\label{coro3.1}
Let $k=\pm1$ and suppose $p_5$ is constant. Then
$f$ is a non-constant entire solution to
\begin{equation}\label{coro3.1-1}
(L(f)+p_5f)^2+\overline{f}^2=1
\end{equation}
if and only if
\begin{equation}\label{coro3.1-2}
f(z)=k\sin (\ug) 
\end{equation}
where $g$ is a non-constant entire function such that $L(\underline{g})^2=1-(p_5)^2$, and either
$(e^{i(g+\underline{g})}-ikp_5)^2=1-(p_5)^2$ or $(e^{i(g-\underline{g})}-ikp_5)^2=1-(p_5)^2$.
\end{corollary}
\begin{proof} ($\Rightarrow$)
Clearly, $a_1=a_2=k$, $b_1=-i$, $b_2=i, D=1$, $d_1=-p_5,d_2=0$.
It follows Theorem \ref{T1} that \eqref{coro3.1-2} holds and \eqref{E1.7} turns to be
\begin{equation*}
\alpha e^{ig}+\beta e^{-ig}=\gamma e^{i\underline{g}}+\delta e^{-i\underline{g}},
\end{equation*}
where $\alpha=\beta=i$ and $\gamma= ikL(\ug)-d_1k,\delta=ikL(\ug)+d_1k$. We claim that $\delta\gamma\not\equiv 0$ whatever $d_1=0$ or not.
Indeed, if $d_1=0$, one of $\delta,\gamma$ is equal to zero implies that $L(g)\equiv 0$, then $\gamma\equiv \delta \equiv 0$ and
$$
e^{2ig}=-\frac{\beta}{\alpha}=-1.
$$
It follows that $g$ is constant. So is $f$. This is a contradiction. For the case $d_1\not=0$, if $\gamma=0$, we have $\delta=2d_1k\not=0$ and
$$e^{i(g+\ug)})+e^{-i(g-\ug)}=-2id_1k.$$
Set $F=e^{i(g+\ug)}$ and $G=e^{-i(g-\ug)}$. By the second main theorem, we get
\begin{equation*}
\begin{split}
T(r,F)&\leq N(r,F)+N\left(r,\frac{1}{F}\right)+N\left(r,\frac{1}{F+2id_1k}\right)+o(T(r,F))\\
&\leq N\left(r,\frac{1}{G}\right)+o(T(r,F))=o(T(r,F)),
\end{split}
\end{equation*}
and similarly $T(r,G)=o(T(r,G))$, both hold for all large $r$ outside a set of finite Lebesgue measure. Therefore, $F$ and $G$ must be constant, so
$g$ is constant, which is impossible. When $d_1\not=0$ and $\delta=0$, a contradiction follows from the similar argument again. Therefore, we know $\alpha\beta\gamma\delta\not\equiv0$.

Hence, by Lemma \ref{new2}, we obtain that $L(g)$ is a polynomial, and
\begin{equation*}
e^{i(g+\underline{g})}=\frac{\delta}{\alpha}=kL(\ug)+ikp_5 \quad\text{or}\quad e^{-i(g-\underline{g})}=\frac{\delta}{\beta}=kL(\ug)+ikp_5.
\end{equation*}
Further, $\alpha\beta-\gamma\delta\equiv 0$ deduces $(L(\ug))^2=1-(d_1)^2=1-(p_5)^2$.\vskip 2mm

($\Leftarrow$)
If $g$ is a non-constant entire function, then $f=k\sin(\underline{g})$ is also non-constant entire, and further
$$
L(f)=k\cos(\ug)L(\ug).
$$
Now we verify that $f$ is a solution to \eqref{coro3.1-1}.
When $e^{i(g+\ug)}=kL(\ug)+ikp_5$, due to $(L(\ug))^2=1-(p_5)^2$, we have
$$e^{ig}=(kL(\ug)+ikp_5)e^{-i\ug},\quad e^{-ig}=(kL(\ug)-ikp_5)e^{i\ug},$$
which implies that
$\sin g=-kL(\ug)\sin(\ug)+kp_5\cos(\ug)$. Similarly if
$e^{-i(g-\ug)}=kL(\ug)+ikp_5$, we have
$\sin g=kL(\ug)\sin(\ug)-kp_5\cos(\ug)$. Combining these considerations, we obtain
\begin{equation*}
\begin{split}
(L(f)&+p_5f)^2+\overline{f}^2\\
&=\left(k\cos(\ug)L(\ug)+kp_5\sin(\ug)\right)^2+\left(\sin(g)\right)^2\\
&=\left(k\cos(\ug)L(\ug)+kp_5\sin(\ug)\right)^2+\left(kL(\ug)\sin(\ug)-kp_5\cos(\ug)\right)^2\\
&=(L(\ug))^2+(p_5)^2=1.
\end{split}
\end{equation*}
\end{proof}

\vskip.05in

Next corollary shows a necessary-sufficient condition for a solution to have a complete representation. The necessary part of the corollary also gives a correction to \cite[Theorem 3.1]{zhengXu-2022} and \cite[Theorem 1.1]{xu-cao-2020}. Example \ref{3rdEx} below shows \cite[Theorem 3.1]{zhengXu-2022} is wrong. The constant $B$ in \cite[Theorem 1.1]{xu-cao-2020}
cannot be defined when $c_2=0$.

\begin{corollary}\label{Cor3.8}
Let $n=2, j=1$, $c=(c_1,c_2)\in \C^2\setminus\{0\}$, and $p\equiv 1$. If $p_m$ ($m=1,\cdots,4$) are constant such that $D\not=0$, then  $f$
is a transcendental entire solution with finite order to the equation
\begin{equation}\label{E1.1-new}
\left(p_{1}f_{ z_1}(z_1,z_2)+p_{2}\overline{f(z_1,z_2)}\right)^{2}+\left(p_{3}f_{z_{1}}(z_1,z_2)+p_{4}\overline{f(z_1,z_2)}\right)^{2}=1
\end{equation}
if and only if
\begin{equation}\label{coro39-10}
f(z)
= \frac{a_1e^{i\underline{g}}-a_2e^{-i\underline{g}}}{2i{D}},\quad \ug_{z_1}=A \quad \mbox{and} \quad
e^{i(g-\underline{g})}=\frac{b_2}{ia_2A},
\end{equation}
where  $A=\pm\sqrt{\frac{b_1b_2}{a_1a_2}}$, $g=A z_1+g_*(z_2)$, and $g_*$ is a polynomial in $z_2$ only. Furthermore, when $c_2\not=0$, $g_*(z_2)$ is a linear function in $z_2$ only.
\end{corollary}
\begin{proof} ($\Rightarrow$)
Since all $a_1,a_2,b_1,b_2,D$ are constants and $p=1$, we get the form of $f$ in (\ref{coro39-10}) by Theorem \ref{T1}, which implies $T(r, e^{i\ug})= O(T(r, f))$ for all large $r$. Here, $g$ satisfies one of properties (b),(c) and (d) in Theorem \ref{T1}. If $g$ satisfies property (d), i.e., if $g_{z_1}$ is transcendental, then $\ug$ is transcendental. This means that $e^{i\ug}$ has infinite order, which contradicts with the hypothesis that $f$ is of finite order, so property (d) is ruled out.
If $g$ satisfies property (b), then $g_{z_1}\equiv0$. Consequently, $\gamma\equiv \delta\equiv 0$. Therefore,
\eqref{E1.7} shows $g$ is constant, so is $f$. This is a contradiction. Thus, $g$ only satisfies property (c), which is derived from {\it Case B} in the second part of proof for Theorem \ref{T1}. More precisely, we have
$\alpha\beta-\gamma\delta\equiv 0$, that is $(p_1^2+p_3^2)\underline{g}_{z_1}^2-(p_2^2+p_4^2)=0$ due to $D\not=0$,
\begin{align}\label{coro39-14}
e^{i(g-\underline{g})}=\frac{\beta}{\delta} =\frac{b_2}{ia_2\underline{g}_{z_1}} \quad \mbox{and} \quad
e^{-i(g-\underline{g})}=\frac{\alpha}{\gamma}=-\frac{b_1}{ia_1\underline{g}_{z_1}}.
\end{align}
So, $g_{z_1}=\underline{g}_{z_1}=A$. Thus,  $g(z_1, z_2)=Az_1+g_2(z_2)$, where
$g_2$ is a polynomial of $z_2$ only. Further, since
$$
g(z_1, z_2)-\underline{g(z_1, z_2)}=Ac_1+g_2(z_2)-g_2(z_2-c_2)
$$
is constant, consequently, if $c_2\neq0$, then $g_2$ is linear, so is $g(z_1, z_2)$. Hence, \eqref{coro39-10} holds.

($\Leftarrow$) It is clear from Lemma \ref{equiv} that \eqref{E1.1-new} has the same solutions as the equation
\begin{eqnarray}\label{coro39-11}
\left(\begin{array}{c}
    \overline{f}  \\ f_{z_1}
\end{array}\right)
&=& \frac{1}{2iD}\left(
  \begin{array}{cc}
   {a_1} & -{a_2}  \\ -{b_1} & {b_2}
  \end{array}
\right)\left(
  \begin{array}{c}
   e^{ig} \\  e^{-ig}
  \end{array}
\right)
\end{eqnarray}
has. Since $g_{z_1}=\underline{g}_{z_1}=A$, we have $\alpha\beta-\gamma\delta=0$. Thus, \eqref{coro39-14} holds since  $e^{i(g-\underline{g})}=b_2/(ia_2A)$.
By \eqref{coro39-10} and \eqref{coro39-14}, we yield
$$
f_{z_1}=\frac{a_1e^{i\underline{g}}ig_{z_1}+a_2e^{-i\underline{g}}ig_{z_1}}{2i{D}}=
\frac{-b_1e^{i{g}}+b_2e^{-ig}}{2i{D}}.
$$
It follows from \eqref{coro39-10} and the above identity that \eqref{coro39-11} holds.

\end{proof}

The following four examples demonstrate the accuracy of Theorem \ref{T1}.  Without any doubt, the property that $g$ is constant in Theorem \ref{T1} occurs. Now, we show the rest of the properties in Theorem \ref{T1} could happen.

\begin{example}\label{ex3.8}
Let $g(z_1, z_2)$ be any entire function with $g_{z_1}\equiv 0$ and $c=(c_1, 0)\neq (0, 0)\in \C^2$.
Let
$$p_1=iz_1, \quad  p_2=i, \quad  p_3=-3z_1, \quad  p_4 =1\quad \mbox{and} \quad p=4iz_1.$$
Then
$f(z_1, z_2)=\frac{1}{2i}e^{i\underline{g}(z_1, z_2)}+\underline{z_1}e^{-i\underline{g}(z_1,z_2)}$ is an entire solution to the equation \eqref{E1.1} when $L(f)=f_{z_1}$ and $p_5=p_6=0$, the property (b) in Theorem \ref{T1} occurs.
\end{example}
\begin{proof}
Since $g-\ug\equiv 0$,
\begin{align*}
& \left(iz_1f_{ z_{1}}+i\overline{f}\right)^{2}+\left(-3z_1f_{z_{1}}+\overline{f}\right)^{2}
 = 8 z_1^2f_{ z_{1}}^2-8z_1f_{z_1}\overline{f}\\
& =  8 z_1^2e^{-2i\underline{g}}-8z_1e^{-i\underline{g}}(\frac{1}{2i}e^{ig}+z_1e^{-ig})\\
&=4iz_1=p(z_1, z_2).
\end{align*}
Hence $f$ is an entire solution to the equation \eqref{E1.1} when $j=1$ and $p_5=p_6=0$.
By a little calculations, we obtain
\begin{align*}
a_1 & =i(k+3)z_1,  & a_2& =i(k-3)z_1,\\
b_1 & = i(k-1), & b_2 &= i(k+1),\\
\alpha& = 16i(k-1)(z_1-c_1)^2, & \beta & = 64(k+1)z_1(z_1-c_1)^2,\\
\gamma_1 &= 0, & \delta_1 &= -64(k-3)z_1(z_1-c_1)^2,
\end{align*}
and $D=4iz_1.$ Further, when $k=1$,
\begin{align*}\label{1-14-3-0}
f(z)
& = \frac{\underline{a_1}e^{i\underline{g}}-\underline{pa_2}e^{-i\underline{g}}}{2i\underline{D}}
= \frac{i(k+3)\underline{z_1}e^{i\underline{g}}-4i\underline{z_1}i(k-3)\underline{z_1}e^{-i\underline{g}}}{2i4i\underline{z_1}}\\
&=\frac{1}{2i}e^{i\underline{g}(z_1, z_2)}+\underline{z_1}e^{-i\underline{g}(z_1,z_2)}.
\end{align*}
When $k=1$, $f$ has the form \eqref{1-14-3} and satisfies the equation \eqref{E1.1} with $L(f)=f_{z_1}$ and $p_5=p_6=0$, and the property (b) in Theorem \ref{T1} exists. Further, when $k=1$,
 we have $\alpha\equiv\gamma\equiv 0$, which means {\it Sub-case A.3} in the proof of Theorem \ref{T1} occurs.
 The solution $f$ could have finite or infinite order depending on how we choose $g$.
\end{proof}

\begin{remark}\label{rem3.9}
We can verify that $f_*(z_1, z_2)=\frac{1}{2}e^{i\underline{g}(z_1, z_2)}-i\underline{z_1}e^{-i\underline{g}(z_1,z_2)}$ is also a solution to the equation  \eqref{E1.1} in the setting of Example \eqref{ex3.8}. Clearly, $f$ in Example \eqref{ex3.8} and $f_*$ here are of two different representations, which can be observed by Lemma \ref{equiv} and Remark \ref{rem2.3}.
Indeed, $g_*=g+\pi/2$.
\end{remark}

\begin{example}\label{3rdEx}
Let $m$ be a positive integer,  $g(z_1, z_2)=z_1+z_2^m+1$ and $ c=(2\pi, 0) \in \C^2$.
Set
$$p_1(z_1,z_2)=p_3(z_1,z_2)=p_4(z_1,z_2)=1 \quad \mbox{and} \quad p_2(z_1, z_2)=-1.
$$
Assume $p(z_1,z_2)=p $ is any non-zero constant.
Then $g_{z_1}=1$,  $g -\underline{g}=2\pi$ and further,
$$f(z_1, z_2)=\frac{1}{4i}\left\{(1-i)e^{i\underline{g}(z_1, z_2)}-(1+i)pe^{-i\underline{g}(z_1,z_2)}\right\}$$ is an entire solution to the equation \eqref{E1.1} when $L(f)=f_{z_1}$ and $p_5=p_6=0$, and the property (c) in Theorem \ref{T1} exists.
\end{example}

\begin{proof} Since $e^{\pm i\underline{g}}=e^{\pm i g}$, $\overline{f}=f$. Thus,
\begin{align*}
& \left(f_{ z_{1}}-\overline{f}\right)^{2}+\left(f_{z_{1}}+\overline{f}\right)^{2}
 = 2(f_{ z_{1}}^2+\overline{f}^2)\\
& =-\frac{1}{8}\left\{(1-i)ie^{ig}+(1+i)ipe^{-ig}\right\}^2-\frac{1}{8}\left\{(1-i)e^{ig}-(1+i)pe^{-ig}\right\}^2\\
&=p=p(z_1, z_2).
\end{align*}
Now, we verify that the conditions in the property (c) of Theorem \ref{T1} hold. Clearly, $g_{z_1}=1$ is a non-zero polynomial and $g-\ug=2\pi$ is constant. Further,
\begin{align*}
a_1 & =k-i,  & a_2& =k+i,
& b_1 & = -(k+i), & b_2 &= -(k-i),
\end{align*}
and $D=2$. It is easy to verify that when $k=1$, $f$ has the form \eqref{1-14-3} and
$$\deg_{z_1}\underline{g}_{z_1}=\deg_{z_1}b_1 -\deg_{z_1}a_1 =\deg_{z_1}b_2 -\deg_{z_1}a_2=0.
$$
Further, the solution $f$ is of finite order $m$.
\end{proof}

\begin{example}\label{4-1ex}
Let $k=\pm 1, \ g(z_1, z_2)=z_2+z_1e^{-2iz_2},\ c=(0, \frac{\pi}2)\in \C^2$, and $p(z_1, z_2)=-2z_2$. Set
$$
p_1(z_1, z_2)=kz_2,\quad  p_2(z_1, z_2)=\frac{k}2, \quad p_3(z_1, z_2)=iz_2, \quad p_4(z_1, z_2)=-\frac{i}2.
$$
Then 
$f(z_1, z_2)=ke^{i\underline{g}(z_1, z_2)}$ is a transcendental entire solution to \eqref{E1.1} when $L(f)=f_{z_1}$
and the property (d) in Theorem \ref{T1} occurs when $a_2\equiv b_1\equiv 0$.
\end{example}

\begin{proof}
Since
$\displaystyle
\underline{g}=z_2-\frac{\pi}2+z_1e^{-2i(z_2-\pi/2)}=z_2-\frac{\pi}2-z_1e^{-2iz_2},
$
thus,
\begin{align*}
& \left(kz_2f_{ z_{1}}+\frac{k}2\overline{f}\right)^{2}+\left(iz_2f_{z_{1}}+\frac{-i}2\overline{f}\right)^{2}
 =2z_2f_{z_1}\overline{f}
 = 2z_2e^{i\underline{g}}i\underline{g}_{z_1}e^{ig}\\
& = 2iz_2e^{i(g+\underline{g})}\underline{g}_{z_1}=2iz_2e^{i(2z_2-\pi/2)}(-e^{-2iz_2})\\
&=-2z_2=p(z_1, z_2).
\end{align*}
This means that $f$ satisfies \eqref{E1.1} with $L(f)=f_{z_1}$ and $p_5=p_6=0$.
Now, we verify that the conditions in the property (d) of Theorem \ref{T1} hold. Clearly, $g$ and $g_{z_1}$ are transcendental entire functions. Also $g+\ug=2z_2-\frac{\pi}2$. Furthermore,
$$
a_1 =2z_2, \quad   a_2 =0, \quad
 b_1  = 0,\quad  b_2 =1\quad \mbox{and} \quad D=-kiz_2.
 $$
It is easy to verify that $f$ has the form \eqref{1-14-3} when $k=\pm 1$.
It follows that the property (d) in Theorem \ref{T1} hold since $a_2\equiv b_1\equiv0$.
\end{proof}

\begin{example}\label{4-2ex}
Let $k=\pm 1, \ g(z_1, z_2)=z_2+z_1e^{2iz_2},\  c=(0, -\frac{\pi}2) \in \C^2$, and $p(z_1, z_2)=-2iz_2$. Set
$$p_1=\frac{1}{2},\quad p_2(z_1, z_2)=-iz_2, \quad p_3=-\frac{ki}{2} \quad \mbox{and}\quad p_4(z_1, z_2)=kz_2.
$$
Then
$f(z_1, z_2)=e^{-i\underline{g}(z_1, z_2)}$ is a transcendental entire solution to \eqref{E1.1} when $L(f)=f_{z_1}$ and $p_5=p_6=0$,
and the property (d) in Theorem \ref{T1} occurs when $a_1\equiv b_2\equiv 0$.
\end{example}

\begin{proof}
Since
$\displaystyle
\underline{g}=z_2+\frac{\pi}2+z_1e^{2i(z_2+\pi/2)}=z_2+\frac{\pi}2-z_1e^{2iz_2},$ so,
\begin{align*}
& \left(\frac{1}2f_{ z_{1}}-iz_2\overline{f}\right)^{2}+\left(-\frac{ki}2f_{z_{1}}+kz_2\overline{f}\right)^{2}
 =-i2z_2f_{z_1}\overline{f}
 = -i2 z_2e^{-i\underline{g}}(-i)\underline{g}_{z_1}e^{-ig}\\
& = -2z_2e^{-i(g+\underline{g})}\underline{g}_{z_1}=-2z_2e^{-i(2z_2+\pi/2)}(-e^{2iz_2})\\
&=-2iz_2=p(z_1, z_2).
\end{align*}
Thus, $f$ is a solution to \eqref{E1.1} and $g+\ug=2z_2+\pi/2$ is a polynomial.  Also,
$$
a_1 =0,  \quad  a_2 =k, \quad
 b_1  = -2kiz_2, \quad  b_2 =0 \quad \mbox{and} \quad D=kz_2.
 $$
It is easy to verify that $f$ has the form \eqref{1-14-3} when $k=\pm 1$.
It follows that the property (d) in Theorem \ref{T1} hold since $a_1\equiv b_2\equiv0$.
\end{proof}

In next two theorems, we replace $L(f)$ in
Theorem \ref{T1} by $L^*(f)$. It is the special case of $L(f)$ in which all $q_j(j=1,2,\cdots,s)$ are non-zero complex numbers. We use the same notations as we do in Theorem \ref{T1}.  Further with $p_5=p_6=0$, \eqref{E1.1} becomes
\begin{eqnarray}\label{E1.1-2}
\left(p_{1}L^*(f)+p_{2}\overline{f}\right)^{2}+\left(p_{3}L^*(f)+p_{4}\overline{f}\right)^{2}=p.
\end{eqnarray}
Indeed, we obtain two necessary-and-sufficient theorems for $L^*(g)$ to be
transcendental, which strengthen the property (d) in Theorem \ref{T1}. In these two theorems,  the entire solution $f$ to \eqref{E1.1-2} is either
$\displaystyle f(z)=c_*e^{i\ug}$ or $\displaystyle f(z)=p_*e^{-i\ug}$, where $c_*$ is a non-zero complex number and $p_*$ is a non-zero polynomial.

\begin{theorem}\label{1stIFF}
Let $f(z)
= {(\underline{a_1}e^{i\underline{g}}-\underline{pa_2}e^{-i\underline{g}})}/{(2i\underline{D})}$ be an entire solution to \eqref{E1.1-2} as in \eqref{1-14-3}.
Then
$L^*(\ug)$ is transcendental, $p_2$ is a non-zero complex number, $p/p_1$ is a non-zero polynomial, and
$$
L^*(\ug)=-\frac{pp_2i}{p_1}e^{-i(g+\ug)}
$$
if and only if $a_2\equiv b_1\equiv 0$.
\end{theorem}
\begin{proof}
$(\Leftarrow$)
$b_1\equiv 0$ if and only if $\alpha \equiv0$. The condition $a_2\equiv 0$ implies that $\delta \equiv 0$. Further we have $\beta\not\equiv 0$
since $b_2\not\equiv 0$. Thus, $D=-2kip_1p_2$. Applying Lemma \ref{equiv} with
$s(f)=L^*(f)$, $t(f)=\overline{f}$ and $u(f)\equiv 0$ gives
\begin{equation*}
\left(\begin{array}{c}
    L^*(f) \\ [5pt] \overline{f}
\end{array}\right)
= \left(
  \begin{array}{cc}
  0 &  \frac{p}{2p_1}  \\ [5pt] \frac{1}{2p_2} & 0
  \end{array}
\right)\left(
  \begin{array}{c}
   e^{ig} \\ [5pt] e^{-ig}
  \end{array}
\right).
\end{equation*}
Hence, we follow the proof of  \eqref{yeAdd3} in Theorem \ref{T1} to obtain that $p_2$ is a non-zero complex number, $p/p_1$ is a non-zero polynomial and
$$
L^*(\ug)=-\frac{pp_2i}{p_1}e^{-i(g+\ug)}.
$$
Thus, if $L^*(\ug)$ is a polynomial, then $g+\ug$ is constant, consequently, $L^*(g)+L^*(\ug)=0$. By Lemma \ref{const},
$L^*(\ug)\equiv 0$. This is a contradiction. Therefore, $L^*(\ug)$ is transcendental.

($\Rightarrow$)  Clearly, $L^*(\ug)\not\equiv 0$.
First, we assume that $g+\ug$ is constant, say, $c_*$. Thus,  $L^*(\ug)=p_*e^{-ic_*}$ is a polynomial, where
$p_*=-\frac{pp_2i}{p_1}$. Consequently, $L^*(g)$ and $L^*(\ug)$ are polynomials.
Since $g+\ug$ is constant, we have $L^*(g)+L^*({\ug})=0$, which implies $L^*(g)\equiv L^*({\ug})\equiv 0$, by Lemma \ref{const}.
Therefore, it contradicts $L^*(\ug)\not\equiv 0.$

Secondly, we assume that  $g+\ug:=q_*$ is a non-constant entire function and set $L^*(\ug)=p_*e^{-iq_*}$.
Thus, $L^*(g)$ is transcendental, which implies $g$ is transcendental. We claim that $g-\ug$ is a non-constant entire function. Indeed,  If $g-\ug$ is constant, say, $c_*$, then, along with the fact that $\ug +g=q_*$, derives that $g=q_*/2+c_*/2$ is a non-constant entire function and
$$T(r, g+\ug)= T(r, q_*)=T(r, g)+O(1).$$
If $q_*$ is a non-constant polynomial, then $g$ is a polynomial. So is $L^*(g)$. This contradicts $L^*(\ug)=p_*e^{-iq_*}$.
Thus, $q_*$ is transcendental, by Lemma 2.1 $T(r, e^{q*})/T(r, q*)\to \infty$ as $r\to \infty$.
Hence, by $L^*(\ug)=p_*e^{-iq_*}$,
$$
    \frac{T(r, L^*(\ug))}{T(r, q*)}= \frac{T(r, p_*e^{-iq_*})}{T(r, q_*)} \to \infty \quad \mbox{as} \ r \to \infty.
$$
This is a contradiction since, by Lemma 2.4,
$$T(r, L^*(g))=O(T(r,g))=O(T(r, q_*)).$$
Therefore, the claim is proved.
Let
\begin{align*}
\gamma(z)
&= D\left( -\underline{a_1}L^*(\underline{D})+\underline{D} L^*(\underline{a_1})+i\underline{a_1}\underline{D} L^*(\underline{g})\right):= \gamma_*+i\underline{a_1}D\underline{D}p_*e^{-iq_*},\\
\delta(z)
& = D\left(\underline{pa_2} L^*(\underline{D})-\underline{D} L^*(\underline{pa_2})+
   i\underline{D}\underline{pa_2} L^*(\underline{g})\right):=\delta_*+iD\underline{D}\underline{pa_2}p_*e^{-iq_*}.
\end{align*}

It follow from \eqref{E1.7}
that
\begin{equation}\label{E1.7-2}
\alpha e^{ig} +(\beta- i\underline{a_1}D\underline{D}p_*) e^{-ig}=\gamma_*e^{i\underline{g}}
+iD\underline{D}\underline{pa_2}p_*e^{-i(g+2\ug)}+
\delta_* e^{-i\underline{g}}.
\end{equation}
All coefficients in \eqref{E1.7-2} are polynomials.
To apply Lemma \ref{borel} to  \eqref{E1.7-2}, we need that $g+3\underline{g}$ is a non-constant function. Indeed, if $g+3\underline{g}$ is a complex number $c_*$,
then
$$ L^*(\underline{g})=-\frac{pp_2i}{p_1}e^{-i(g+\underline{g})}=-i\frac{pp_2}{p_1}e^{-ic_*}e^{-4i\underline{g}}.$$
Thus, $T(r,e^{-4i\underline{g}})=O(T(r,\underline{g}))+O(\log r)$. We get a contradiction by Lemma 2.1 and the fact that $g$ is transcendental. Hence, all coefficients in \eqref{E1.7-2}  are zero by Lemma \ref{borel} since $g$, $g\pm \ug$ and
$g+3\ug$ are non-constant entire functions.
Thus, we get $b_1\equiv a_2\equiv 0$.
\end{proof}


\begin{theorem}\label{2stIFF}
Let $f(z)
= {(\underline{a_1}e^{i\underline{g}}-\underline{pa_2}e^{-i\underline{g}})}/{(2i\underline{D})}$  be an entire solution to \eqref{E1.1-2} as in \eqref{1-14-3}. If $ L^*(\ug)\not \equiv 0$, then $ L^*(\ug)$ is transcendental, $p_1$ is a non-zero complex number, $p/p_2$ is a non-zero polynomial, and
\begin{equation}\label{2}
 L^*(\ug)=\frac{i\underline{p_2}}{\underline{p} p_1}e^{i(g+\ug)}-\frac{i  L^*(\underline{p})}{\underline{p}} +\frac{i L^*(\underline{p_2})}{\underline{p_2}}
\end{equation}
if and only if $a_1\equiv b_2\equiv 0$.\end{theorem}
\begin{proof}
$(\Leftarrow)$ The condition $a_1\equiv b_2\equiv 0$ implies $a_2=2kp_1, b_1=2kp_2$ and $D=2kip_1p_2$. We obtain from Lemma \ref{equiv} that
\begin{equation*}
\left(\begin{array}{c}
   L^*(f) \\ [5pt] \overline{f}
\end{array}\right)
= \left(
  \begin{array}{cc}
  \frac1{2p_1}  & 0  \\ [5pt] 0& \frac{p}{2p_2}
  \end{array}
\right)\left(
  \begin{array}{c}
   e^{ig} \\ [5pt] e^{-ig}
  \end{array}
\right).
\end{equation*}
Hence, $p_1$ is constant, $p/p_2$ is a polynomial,
$$f=\frac{\underline{p}}{2\underline{p_2}}e^{-i\underline{g}} \quad \mbox{and} \quad L^*(f)=\frac{1}{2p_1}e^{ig}.$$
Taking the operator $ L^*$ on $f$ and comparing with $ L^*(f)$ above, we obtain \eqref{2} as we do in \eqref{wang-14-n} in
Theorem \ref{T1}. The rest of the proof is similar to the proof of Theorem \ref{1stIFF}.
\vskip.1in
$(\Rightarrow)$ If $ L^*(\underline{g})$ has the form \eqref{2}, then we write
$$\gamma=\gamma^*-\frac{\underline{a_1}D\underline{D}\underline{p_2}}{\underline{p}p_1}e^{i(g+\underline{g})},\quad \delta=\delta^*-\frac{\underline{a_2}D\underline{D}\underline{p_2}}{p_1}e^{i(g+\underline{g})},$$
where $\gamma$ and $\delta$ are defined in the proof of Theorem \ref{T1}; and $\gamma^*$ and $\delta^*$ are rational functions. It follows from \eqref{E1.7} that
\begin{equation}\label{3}
(\alpha +\frac{\underline{a_2}D\underline{D}\underline{p_2}}{p_1})e^{ig}+\beta e^{-ig}=\gamma^* e^{i\underline{g}}-\frac{\underline{a_1} D\underline{D}\underline{p_2}}{\underline{p}p_1}e^{i(g+2\underline{g})}+\delta_* e^{-i\underline{g}}.
\end{equation}
Similar to the proof of Theorem \ref{1stIFF}, we know that $g$, $g\pm \underline{g}$ and $g+3\ug$ are non-constant.
Hence, \eqref{3} satisfies all hypothesis of Lemma \ref{borel} and all coefficients in \eqref{3} are zero. This means $a_1\equiv b_2\equiv 0$.
\end{proof}

\section{Case II.\, $D\equiv 0$ and $d_2\not\equiv 0$}
\begin{theorem} \label{T17}
Let $p,p_1,p_2,\cdots,p_6$ be defined as in Theorem \ref{T1}, and further $D\equiv 0$ and $d_2\not\equiv 0$. Then every entire solution to \eqref{E1.1}
on $\mathbb{C}^{n}$ has the form:
\begin{equation}\label{1-14-37}
f(z)
=\frac{pa_2e^{-ig}-a_1e^{i g}}{2id_2},
\end{equation}
where $g$ is an entire function having one of the following properties:
\begin{enumerate}
\item[(i)] $g$ is constant.
\item[(ii)] $L(g)$ is a polynomial, and either $\overline{g}+g$ or $\overline{g}-g$ is constant if $d_1\not\equiv 0$, while $a_1a_2\not\equiv 0$ if $d_1\equiv 0$.
\end{enumerate}
In addition, if $L(f)=f_{z_j}$ and $d_1\equiv D\equiv 0$ when $j\in\{1, \cdots, n\}$, then we have \eqref{1-14-37},
where $g$ is an entire function having one of the following properties:
\begin{enumerate}
\item[(a)] $g$ is constant;
\item[(b)] $g_{z_j}\equiv 0$ when either $\deg_{z_j}\tilde{b}_1<\deg_{z_j}a_1$ or $\deg_{z_j}\tilde{b}_2<\deg_{z_j}a_2$, where  $\tilde{b}_1\stackrel{def}{=}kp_5-ip_6$ and $\tilde{b}_2=kp_5+ip_6$;
\item[(c)] $g_{z_j}$ is a non-zero polynomial when
$$\deg_{z_j}g_{z_j}=\deg_{z_j}\tilde{b}_1- \deg_{z_j}a_1=\deg_{z_j}\tilde{b}_2- \deg_{z_j}a_2\ge 0,
$$
where $\tilde{b}_1$ and $\tilde{b}_2$ are defined in (b).
\end{enumerate}
\end{theorem}

\begin{proof}
Let $f$ be an entire solution to \eqref{E1.1}. Since $D\equiv 0$ and $d_2\not\equiv 0$, we rewrite \eqref{E1.1} as
$$
(p_1L(f)+p_5f+p_2\of)^2+(p_3L(f)+p_6f+p_4\of)^2=p.
$$
Applying Lemma \ref{equiv} with
$$s(f)=L(f), \quad t(f)=f\quad \mbox{and} \quad u(f)= \overline{f},$$
we obtain
\begin{align}\label{E1.27}
\left(\begin{array}{c}
    L(f)  \\ [5pt] f
\end{array}\right)
= \frac{1}{2id_2}\left(
  \begin{array}{cc}
    {\tilde{b}_1} & -p{\tilde{b}_2} \\ [5pt] -a_1 & pa_2
  \end{array}
\right)\left(
  \begin{array}{c}
   e^{ig} \\ [5pt]  e^{-ig}
  \end{array}
\right)+
 \frac{1}{d_2} \left(
  \begin{array}{c}
  d_1 \\ [5pt] 0
  \end{array}
\right) \overline{f},
\end{align}
where ${\tilde{b}_1}=kp_5-ip_6$ and $ {\tilde{b}_2}=kp_5+ip_6$.
 It follows that
$$
f=
\frac{1}{2id_2}\left(
  -a_1 \  pa_2
    \right)
    \left(
  \begin{array}{c}
   e^{ig} \\  e^{-ig}
  \end{array}
\right),
$$
which is just the representation \eqref{1-14-37}. Similarly as in the proof of Theorem \ref{T1}, we obtain
 \begin{align}\label{9.17}
L(f)
=
\frac{1}{2i(d_2)^2}\left(
  \begin{array}{c}
    L(d_2)a_1-d_2L(a_1)-
    d_2iL(g)a_1\\ [5pt]
    -L(d_2)pa_2+d_2L(pa_2)-
    d_2iL(g)pa_2
  \end{array}
\hskip-.05in\right)^T
\hskip-.05in
\left(
  \begin{array}{c}
   e^{ig} \\ [5pt] e^{-ig}
  \end{array}
\right).
\nonumber
 \end{align}
On the other hand, \eqref{E1.27} shows
\begin{equation}\label{yenew277}
L(f)=\frac{1}{2id_2}\left(
   \tilde{b}_1\quad  -p\tilde{b}_2
  \right)\left(
  \begin{array}{c}
   e^{ig} \\  e^{-ig}
  \end{array}
\right)+\frac{d_1}{d_2}\overline{f}.
\end{equation}
Taking the representations of $f$ and $L(f)$ into \eqref{yenew277} yields
\begin{align*}\label{yenew3}
&  d_2\overline{d_2} \left(
   \tilde{b}_1 \quad  -p\tilde{b}_2
  \right)\left(
  \begin{array}{c}
   e^{ig} \\  e^{-ig}
  \end{array}
\right)+d_1d_2 \left(
  -\overline{a_1} \ \,  \overline{pa_2}
    \right)
    \left(
  \begin{array}{c}
   e^{i\overline{g}} \\  e^{-i\overline{g}}
  \end{array}
\right) \\ \nonumber
 & \qquad = \overline{d_2}\left(
  \begin{array}{c}
     L(d_2)a_1-d_2L(a_1)-
    d_2iL(g)a_1\\ [5pt]
    -L(d_2)pa_2+d_2L(pa_2)-
    d_2iL(g)pa_2
  \end{array}
\hskip-.05in\right)^T
\hskip-.05in
\left(
  \begin{array}{c}
   e^{ig} \\ [5pt] e^{-ig}
  \end{array}
\right),
\end{align*}
which can be written as
\begin{equation}\label{E1.77}
\alpha e^{i\overline{g}}+\beta e^{-i\overline{g}}=\gamma e^{ig}+\delta e^{-ig},
\end{equation}
where
\begin{eqnarray*}
\alpha(z)& = & -d_1d_2\overline{a_1}, \qquad  \qquad \beta(z) =d_1d_2\overline{pa_2}, \\ 
\gamma(z)
&=& \overline{d_2}\left[L(d_2)a_1-d_2(L(a_1)+
    iL(g)a_1+\tilde{b}_1)\right],\\
\delta(z)
& =& \overline{d_2}\left[\,-L(d_2)pa_2+d_2(L(pa_2)-
    iL(g)pa_2+p\tilde{b}_2)\right].
\end{eqnarray*}
Here, $\alpha$ and $\beta$ cannot be identical to zero at the same time since $d_2\not\equiv 0$.\vskip 1mm
\par Now we continue the argument under the assumption $d_1\not\equiv 0$. Similarly as in the proof of Theorem \ref{T1}, we consider the following three cases.\vskip 1mm
\par
{\it Case 1:}\,exactly two of $\alpha,\beta$, $\gamma$, and $\delta$ are identical to zero.\vskip 1mm
\par
{\it Sub-case 1.1}:\,$\alpha\equiv \gamma\equiv 0$  and $\beta\delta\not\equiv 0$. Then $a_1\equiv \tilde{b}_1\equiv 0$, a contradiction due to $d_2\not\equiv 0$.\vskip 1mm
\par
{\it Sub-case 1.2}:\,$\alpha\equiv \delta\equiv 0$ and $\beta\gamma\not\equiv 0$. Thus, we have $a_1\equiv 0, a_2=2kp_1\not\equiv 0$. Then $\delta\equiv 0$ implies
\begin{equation*}
L(g)=i\frac{L(d_2)}{d_2}-i\frac{L(pa_2)}{pa_2}-i\frac{\tilde{b}_2}{a_2},
\end{equation*}
which means $L(g)$ is a polynomial. So is $\gamma$. Further,
 \eqref{E1.77} shows
 $ e^{i(\overline{g}+g)}=\beta/\gamma$,
then $\overline{g}+g$ is constant. Thus, $g$ has the property (ii) in Theorem \ref{T17}.\vskip 1mm
\par {\it Sub-case 1.3}:\,$\beta\equiv \gamma \equiv 0$ and $\alpha\delta\not\equiv0$. Thus, $a_2 \equiv 0$ and $a_1=2kp_1\not\equiv 0$.
Clearly, $\gamma \equiv 0$ implies
\begin{equation*}
L(g)=i\frac{L(a_1)}{a_1}-i\frac{L(d_2)}{d_2}+i\frac{\tilde{b}_1}{a_1},
\end{equation*}
which means $L(g)$ is a polynomial. So is $\delta$. Further,
 \eqref{E1.77} shows that
 $ e^{-i(\overline{g}+g)}=\alpha/\delta$ is a polynomial,
then $\overline{g}+g$ is constant. We have the property (ii) in Theorem \ref{T17}.\vskip 1mm
\par
{\it Sub-case 1.4}:\, $\beta\equiv \delta \equiv 0$ and $\alpha\gamma\not\equiv0$. Then $ a_2\equiv \tilde{b}_2\equiv 0$, which contradicts with the assumption $d_2\not\equiv 0$.\vskip 1mm
\par
{\it Sub-case 1.5}:\,$\gamma\equiv \delta \equiv 0$ and $\alpha\beta\not\equiv 0.$ Then \eqref{E1.77} gives
$e^{-2i\overline{g}}=-\alpha/\beta$. It follows from $\alpha/\beta$ is a polynomial that $g$ is constant. Thus, $g$ has the property
(i) in Theorem \ref{T17}.\par

\par
{\it Case 2:}\, exactly one of $\alpha,\beta,$ $\gamma$ and $\delta$ equals identically to zero. This is a straightforward
application of Lemma \ref{new3} with $u=L(g)$. Thus, $g$ has the property (i) in Theorem \ref{T17}.
\vskip 2mm
\par
{\it Case 3:}\, $\alpha\beta\gamma\delta\not\equiv 0$. Similarly as in Sub-case 1.3 in the proof of Theorem \ref{T1}, we know that $L(g)$ is a polynomial and $\overline{g}+g$ or $\overline{g}-g$ is constant. Therefore, in this sub-case, $g$ has the property (ii) in Theorem \ref{T1}.\vskip 2mm
\par Next, we continue the argument under the assumption $d_1\equiv 0$. Hence, \eqref{E1.77} becomes
$\gamma_1 e^{2ig}=-\delta_1$, where
\begin{equation*}
\begin{split}
\gamma_1(z)&=L(d_2)a_1-d_2(L(a_1)+
    iL(g)a_1+\tilde{b}_1),\\
\delta_1(z)& =-L(d_2)pa_2+d_2(L(pa_2)-
    iL(g)pa_2+p\tilde{b}_2).
\end{split}
\end{equation*}
To finish the proof, we assume that $g$ is a non-constant entire function and claim that $\gamma_1\equiv 0$. Indeed, if
$\gamma_1\not\equiv 0$, then $e^{2ig}=-\delta_1/\gamma_1$.
Since the characteristic function of any rational function is $O(\log r)$,  applying Lemma \ref{L5} to the above equation, we yield that
$$T(r,e^{2ig})=O(T(r,L(g) ))+O(\log r)=O(T(r, g))+O(\log r),$$
for all $r$ possibly outside a set of finite Lebesgue measure. The equality leads to a contradiction by Lemma \ref{L1}. Thus the claim is proved. \par

Further, we claim that $a_1\not\equiv 0$. Indeed, if $a_1\equiv 0$, then by the previous claim $\gamma_1\equiv -d_2\tilde{b}_1\equiv 0$, which means $\tilde{b}_1\equiv 0$. Consequently, $d_2\equiv 0$. This is a contradiction.
It follows from $\gamma_1\equiv 0$ and $a_1\not\equiv 0$ that $L(g)$ is a polynomial and
$$
L(g)=\frac{i\tilde{b}_1}{{a_1}} +\frac{iL(a_1)}{{a_1}}-\frac{iL(d_2)}{{d_2}}.
$$
At the same time, we also have $\delta_1 \equiv 0$ by $\gamma_1\equiv 0$. Similarly, we know $a_2\not\equiv 0$, and further we have
$$
L(g)=-\frac{i\tilde{b}_2}{{a_2}} -\frac{iL(pa_2)}{pa_2}+\frac{iL(d_2)}{d_2}.
$$
Therefore, $L(g)$ has the property (ii) and the first part of the theorem is proved.\vskip 1mm
\par
When $L(f)=f_{z_j}$ and $d_1\equiv D\equiv 0$, if $g$ is non-constant, by the above proof of (ii), we get
$$
{g}_{z_j}= \frac{i\tilde{b}_1}{{a_1}} +\frac{i({a_1})_{z_j}}{{a_1}}-\frac{i(d_2)_{z_{j}}}{d_2}
\quad \mbox{and} \quad
{g}_{z_j}= -\frac{i\tilde{b}_2}{{a_2}} -\frac{i({pa_2})_{z_j}}{{pa_2}}+\frac{i(d_2)_{z_{j}}}{d_2}.
$$
Now, we regard $a_1, \tilde{b}_1, a_2$ and $\tilde{b}_2$ as polynomials of $z_j$. Then there are polynomials $S_m$ and $T_m$ ($m=1,2$) in $z_j$,  whose coefficients are polynomials in
$z_1, \cdots, z_{j-1}, z_{j+1}, \cdots, z_n$, such that
$$
i\tilde{b}_1=S_1a_1+T_1 \qquad \mbox{and} \qquad i\tilde{b}_2=S_2a_2+T_2
$$
where $\deg_{z_j}T_1<\deg_{z_j}a_1$ and $\deg_{z_j}T_2<\deg_{z_j}a_2$. Hence,
$$
{g}_{z_j}-S_1= \frac{T_1}{{a_1}} +\frac{i({a_1})_{z_j}}{{a_1}}-\frac{i(d_2)_{z_{j}}}{d_2}
$$
and
$$
{g}_{z_j}-S_2= -\frac{T_2}{{a_2}} -\frac{i({pa_2})_{z_j}}{{pa_2}}+\frac{i(d_2)_{z_{j}}}{d_2}.
$$
Let $z_1, \cdots, z_{j-1}, z_{j+1}, \cdots, z_n$ be fixed and let $z_j$ go to infinity. Thus, $g_{z_j}-S_m\equiv 0$ by Liouville's theorem.
In particular,
if
$$\deg_{z_j}\tilde{b}_1<\deg_{z_j}a_1\quad \text{or}\quad \deg_{z_j}\tilde{b}_2<\deg_{z_j}a_2,$$ then $g_{z_j}\equiv 0$.
Therefore, $g$ has the property (b) in our theorem.
If $\deg_{z_j}\tilde{b}_1\ge \deg_{z_j}a_1$ and $\deg_{z_j}\tilde{b}_2\ge \deg_{z_j}a_2$, then
$g_{z_j}$ is a polynomial with
$$
\deg_{z_j}g_{z_j}=\deg_{z_j}\tilde{b}_1-\deg_{z_j}a_1=\deg_{z_j}\tilde{b}_2-\deg_{z_j}a_2.
$$
Therefore, $g$ has the property $(c)$ and the theorem is proved completely.
\end{proof}

Our next corollary improves the results in \cite[Theorems 2.2 and 2.3]{xuXu-2022} due to our matrix method.

\begin{corollary}\label{Cor4.2}
Let $n=s=2$ and $q_1=q_2=p=1$. If $p_m (m=1,2,\cdots, 6)$ are constant and $p_2=p_4=0$ in Theorem \ref{T17}, then
\begin{eqnarray}\label{E1.131}
\left(p_{1}L(f)+p_5f\right)^{2}+\left(p_{3}L(f)+p_6f\right)^{2}=p\end{eqnarray}
has a non-constant entire solution $f$ if and only if $a_1a_2\not\equiv 0$, $a_1\tilde{b}_2+a_2\tilde{b}_1=0$, and
$$
f(z)
=\frac1{2i{d_2}}\left(a_2e^{-i{g}}-a_1e^{ig}\right),
$$
where $g(z_1, z_2)=\frac{ib_1}{a_1}z_1+g^*(z_2-z_1)$, $g^*$ is a polynomial of one variable, and $d_2, a_1, a_2, \tilde{b}_1,\tilde{b}_2$ are defined in Theorem \ref{T17}.
\end{corollary}
\begin{proof}
All $a_1,a_2,\tilde{b}_1,\tilde{b}_2$ and $d_2(\not=0)$ are constants. At the same time, $p_2=p_4=0$ implies that $b_1=b_2=0$ and $D=d_1=0$. \vskip 1mm
\par
($\Rightarrow$)
By Theorem \ref{T17}, $f(z)=(a_2e^{-i{g}}-a_1e^{i{g}})/(2id_2)$ and $L(g)$ is a polynomial. If $L(g)\equiv 0$, then
$L(f)\equiv 0$. Therefore,   $f$ is constant since \eqref{E1.131}. This is a contradiction. Thus, $L(g)$ is a non-zero
polynomial. Consequently, $g$ is non-constant and we get from the proof of Theorem \ref{T17} under the assumption $d_1\equiv 0$ that $a_1a_2\not\equiv 0$,
$$
L(g)=g_{z_1}+g_{z_2}=\frac{i\tilde{b}_1}{a_1}=-\frac{i\tilde{b}_2}{a_2},
$$
which gives $a_1\tilde{b}_2+a_2\tilde{b}_1=0$ and
$$
g(z_1, z_2)=-\frac{i\tilde{b}_2}{a_2}z_1+g^*(z_2-z_1)=\frac{i\tilde{b}_1}{a_1}z_1+g^*(z_2-z_1),
$$
where $g^*$ is a polynomial of one variable.\vskip 1mm
\par
($\Leftarrow$) By a little calculation, we obtain
$$L(g)=\frac{i\tilde{b}_1}{a_1}=-\frac{i\tilde{b}_2}{a_2}\not=0\quad \mbox{and}\quad
L(f)=-iL(g)\frac{a_1e^{i{g}} +a_2e^{-i{g}}}{2id_2}.
$$
If $f$ is constant, then $L(f)\equiv 0$. Since $L(g)\neq 0$, so, $a_1e^{i{g}} +a_2e^{-i{g}}\equiv 0$, which implies that
$g$ is constant, consequently, $L(g)\equiv 0$. This is a contradiction.
Now, we check that $f$ is a solution to \eqref{E1.131}.  It suffices to verify $f$ and $L(f)$ to satisfy
\eqref{E1.27} with $d_1\equiv 0$ and $p=1$ since equations \eqref{E1.131} and \eqref{E1.27} with $d_1\equiv 0$ and $p=1$ are equivalent. Indeed,
\begin{align*}
\left(\begin{array}{c}
    L(f) \\ [5pt] f
\end{array}\right)
&   = \frac{1}{2id_2}\left(
  \begin{array}{cc}
   -iL(g)a_1  & -iL(g)a_2 \\ [5pt]  -a_1 & a_2
  \end{array}
\right)\left(
  \begin{array}{c}
   e^{ig} \\ [5pt]  e^{-ig}
  \end{array}
\right)\\
& =\frac{1}{2id_2}\left(
  \begin{array}{cc}
   \tilde{b}_1 & -\tilde{b}_2  \\ [5pt]  {-a_1} & a_2
  \end{array}
\right)\left(
  \begin{array}{c}
   e^{ig} \\ [5pt] e^{-ig}
  \end{array}
\right).
\end{align*}
Thus, the corollary is completely proved.
\end{proof}

\vskip.05in
When $p_3=p_5=0$
\,\,in \eqref{E1.131}, we derive two necessary-sufficient results on the form of entire solutions as below.

\begin{corollary}\label{1stCoro}
 Let  $k=\pm1$. Assume $p_1$ and $p$ are non-zero polynomials on $\C^n$.
 If $p_6$ is a non-zero constant and $p$ is irreducible, then $f$ is a transcendental entire solution to
\begin{equation}\label{coro4.2-1}
(p_1L(f))^2+(p_6f)^2=p
\end{equation}
if and only if
\begin{equation}\label{coro4.2-2}
f(z)=\frac{k}{2ip_6}(e^{ig}-pe^{-ig}),
\end{equation}
where $g$ is a non-constant entire function, $L(g)=kp_6/p_1$, $L(p)=0$, and $p_1$ is a non-zero constant.
%
\end{corollary}
\begin{proof} ($\Rightarrow$) Since \eqref{coro4.2-1} is the case when $p_3=p_5=0$ in \eqref{E1.131},
$$d_2=-p_1p_6, \quad a_1=a_2=kp_1, \quad \tilde{b}_1=-ip_6\quad \mbox{and} \quad \tilde{b}_2=ip_6.
$$
By Theorem \ref{T17}, we get \eqref{coro4.2-2}.
Clearly, $f$ is transcendental if and only if $g$ is a non-constant entire function.
Therefore,  we obtain from the proof of Theorem \ref{T17} that $\gamma_1\equiv 0$ and $\delta_1\equiv 0$, which give
$$
L(g)=\frac{p_6}{kp_1}-i\frac{L(p_6)}{p_6}=\frac{p_6}{kp_1}
$$
and
$$
L(g)=\frac{p_6}{kp_1}-i\frac{L(p)}{p}+i\frac{L(p_6)}{p_6}=\frac{p_6}{kp_1}-i\frac{L(p)}{p}.
$$
Hence, $L(g)={kp_6}/p_1$ and $L(p)=0$. Since $L(g)$ is entire, $p_1$ is a non-zero constant function.

($\Leftarrow$)
We already have $L(p_6)=L(p)=0$. When $g$ is a non-constant entire function with $L(g)={kp_6}/p_1$, then $f$ in \eqref{coro4.2-2} is transcendental entire, and
\begin{equation*}
\begin{split}
L(f)&=L\left(\frac{k}{2ip_6}\right)(e^{ig}-pe^{-ig})+\frac{k}{2ip_6}\left(e^{ig}iL(g)-L(p)e^{-ig}+pe^{-ig}iL(g)\right)\\
&=\frac1{2p_1}(e^{ig}+pe^{-ig}).
\end{split}
\end{equation*}
Further, it is straightforward to verify that \eqref{coro4.2-1} is satisfied.
\end{proof}

\begin{corollary}\label{1stCoro'}
 Let $p_1$ and $p_6$ be non-zero polynomials on $\C^n$, $k=\pm1$, and $j\in \{1, \cdots, n\}$. Then $f$ is a non-constant entire solution to
\begin{equation}\label{coro4.2-3}
(p_1f_{z_j})^2+(p_6f)^2=1
\end{equation}
if and only if
\begin{equation}\label{coro4.2-4}
f(z)=\frac{k}{p_6}\sin (g) 
\end{equation}
where $g$ is a non-constant entire function, $g_{z_j}=kp_6/p_1$, $p_6$, and $p_1$ are non-zero constant functions.
%
\end{corollary}
\begin{proof} ($\Rightarrow$) By Theorem \ref{T17}, we easily get \eqref{coro4.2-4}.
If $g$ is a constant function,
then $f$ is constant by  \eqref{coro4.2-4}. Thus, we conclude that $g$ is non-constant.
Similarly as in proof of Corollary \ref{1stCoro}, we have
$$
{g}_{z_j}=\frac{p_6}{kp_1}-i\frac{(p_6)_{z_j}}{p_6}\quad \text{and}\quad {g}_{z_j}=\frac{p_6}{kp_1}+i\frac{(p_6)_{z_j}}{p_6}.
$$
Hence, $g_{z_j}={kp_6}/p_1$ and $(p_6)_{z_j}\equiv 0$. Also,  \eqref{coro4.2-3} implies that $p_1$ and $p_6$ are
co-prime. Therefore, $p_1$ is a non-zero constant function.
Now, we need to prove that $p_6$ is constant.
Indeed, \eqref{coro4.2-4}
implies that $f$ is entire if and only if $(e^{2ig}-1)/p_6:=h$ is entire. Therefore, 
$h_{z_j}=2ie^{2ig}g_{z_j}/p_6$. Hence, we fix $z_1,\cdots, z_{j-1}, z_{j+1}, \cdots, z_n$, treat $h_{z_j}$ as a function of the variable $z_j$ and obtain
$$
h=\int \frac{2ig_{z_j}}{p_6}e^{2ig} dz_j=\frac1{p_6}e^{2ig}+h_*,
$$
where $h_*$ is entire with $(h_*)_{z_j}\equiv 0$.
Combining this with the definition of $h$ shows $ -1=p_6h_*$. Consequently, both $p_6$ and $h_*$ are constant.

($\Leftarrow$)
If $g$ is a non-constant entire function, $(p_6)_{z_j}\equiv 0$ and $g_{z_j}={kp_6}/p_1$, it is straightforward to verify that $f$ in \eqref{coro4.2-4} satisfies the equation \eqref{coro4.2-3}.
\end{proof}

\section{Case III.\,$D\equiv 0,$ $d_1\not\equiv 0$ and $d_2\equiv 0$}
When $D\equiv p_1p_4-p_2p_3\equiv 0$, we have
$$p_1d_1=-p_2d_2,\quad p_3d_1=-p_4d_2.$$
This means that once one of $p_1,p_3$ is non-zero, $d_2\equiv 0$ implies $d_1\equiv 0$, which contradicts with $d_1\not\equiv 0$. Thus in this case,
we know $p_1\equiv p_3\equiv 0$, so, $L(f)$ term does not appear in \eqref{E1.1}, and \eqref{E1.1} degenerates to
\begin{equation}\label{E1.20}
\left(p_{2}\overline{f}+p_5 f\right)^{2}+\left(p_4\overline{f}+p_6f\right)^{2}=p\end{equation}

\begin{theorem} \label{T1.3}
Let $c, p_2, p_5, p_4, p_6, d_1, p, b_1, b_2, \tilde{b}_1$, and $\tilde{b}_2$ be defined as in Theorem \ref{T17}.
Then every entire solution $f$ to \eqref{E1.20} on $\mathbb{C}^{n}$  has the form:
\begin{eqnarray}\label{E1.3-4}
f(z)=\frac{b_1e^{ig}-pb_2e^{-ig}}{2id_1},
\end{eqnarray} where
$g$ is an entire function having one of the following properties:\par
\begin{enumerate}
\item[(i)]  $g$ is constant.\par

\item[(ii)] $g$ is transcendental, and further, $\overline{g}+g$ is constant when either $\tilde{b}_1\equiv b_2\equiv 0$ with
 $p_5$ and $p/p_2$ being non-zero constants,
or $\tilde{b}_2\equiv b_1\equiv 0$ with $p_2$ and $p/p_5$ being non-zero constants.\par
\item[(iii)]
$
(d_1)^2\overline{pb_1b_2}\equiv\overline{d_1}^2p\tilde{b}_1\tilde{b}_2\not\equiv0,
$
and
either $e^{i(\overline{g}+g)}=\frac{\overline{d_1}p\tilde{b}_2}{d_1\overline{b}_1}$ or $e^{i(\overline{g}-g)}=-\frac{\overline{d_1}\,\tilde{b}_1}{d_1\overline{b_1}}$ is constant.

\end{enumerate}
 \end{theorem}

\begin{proof}
Let $f$ be an entire solution to \eqref{E1.20}. Applying Lemma \ref{equiv} with $s(f)=\overline{f}$, $t(f)=f$ and
$u(f)\equiv 0$ gives
\begin{eqnarray}\label{E1.21}
\left(\begin{array}{c}
  \overline{f}\\ [5pt]  f
\end{array}\right)
= \frac{1}{2id_1}\left(
  \begin{array}{cc}
 -\tilde{b}_1   & p\tilde{b}_2 \\ [5pt]  b_1 & -p b_2
  \end{array}
\right)\left(
  \begin{array}{c}
   e^{ig} \\  [5pt] e^{-ig}
  \end{array}
\right),
\end{eqnarray} where $\tilde{b}_2=kp_5-ip_6$,  $\tilde{b}_1=kp_5+ip_6$ and $k=\pm 1$. Therefore,

\begin{equation*}
f(z)=\frac{1}{2id_1}(b_1e^{ig}-pb_2e^{-ig}),
\end{equation*}
which is \eqref{E1.3-4}.
By equaling two expressions of $f$ in \eqref{E1.21} and \eqref{E1.3-4}, we yield
\begin{align*}
  \frac{1}{2i\overline{d_1}}\left(\overline{b_1}e^{i\overline{g}}-\overline{pb_2}e^{-i\overline{g}}\right)
=\frac1{2id_1}\left(-\tilde{b}_1e^{ig}+p\tilde{b}_2e^{-ig}\right),
 \end{align*}
 which gives
\begin{equation}\label{E1.23}
\alpha_2 e^{i\overline{g}}+\beta_2 e^{-i\overline{g}}=\gamma_2 e^{ig}+\delta_2 e^{-ig},
\end{equation}
where $\alpha_2(z) = d_1\overline{b_1},\ \
 \beta_2(z) =-d_1\overline{pb_2},\ \
\gamma_2(z) =-\overline{d_1}\tilde{b}_1$ and
  $\delta_2(z) =\overline{d_1}p\tilde{b}_2$.
If $\alpha_2\equiv \beta_2\equiv 0$, i.e., $b_1\equiv b_2\equiv 0$, then $d_1\equiv 0$, which is impossible. Similarly, $\gamma_2\equiv \delta_2\equiv0$, i.e., $\tilde{b}_1\equiv \tilde{b}_2\equiv 0$, we obtain a contradiction to $d_1\not\equiv 0$ again. \vskip 1mm
\par
If $g$ is constant, then the conclusion (i) in the theorem is already obtained. We assume below that $g$ is non-constant, so $e^{i\overline{g}}$ is transcendental,  one of $\overline{g}+g$ and $\overline{g}-g$ is also non-constant and $\log r=o(T(r, e^{i\overline{g}})).$\par

To proceed the proof, we consider it case-by-case. \vskip 1mm
\par

{\it Case 1:} exactly two of $\alpha_2,\beta_2,$ $\gamma_2$, and $\delta_2$ are identical to zero.\vskip 1mm
\par
{\it Sub-case 1.1:}\,$\alpha_2\equiv \gamma_2 \equiv 0$. We have $b_1\equiv\tilde{b}_1\equiv 0$, which implies $d_1\equiv 0$. This means that \eqref{E1.20} does not have entire solutions.\vskip 1mm
\par
{\it Sub-case 1.2:} $\alpha_2\equiv\delta_2\equiv 0,$ i.e., $b_1\equiv \tilde{b}_2\equiv 0$. We obtain
$d_1=2kip_2p_5$,
$$\overline{f}=\frac{1}{2p_2}e^{ig}\quad \mbox{and}\quad {f}=\frac{p}{2p_5}e^{-ig}.$$
It follows from these equations that
$p_2$ is constant,  $p/p_5$ is a non-zero polynomial, and further,
 $$e^{i(\overline{g}+g)}=\frac{p_2\overline{p}}{\overline{p_5}}.$$
 This derives that $\overline{g}+g$ is constant, $p/p_5$ is constant and $g$ is transcendental. The property (ii) in the theorem is proved.\vskip 1mm
\par
{\it Sub-case 1.3:}\, $\beta_2\equiv \gamma_2 \equiv 0,$ i.e.,
$b_2\equiv \tilde{b}_1\equiv 0$. We have $d_1=-2kip_2p_5, b_1=2kp_2$ and $\tilde{b}_2=2kp_5.$ Using these relations in \eqref{E1.21} gives
$$f(z)=\frac{1}{2p_5}e^{ig} \qquad \mbox{and} \qquad
\overline{f}=\frac{p}{2p_2}e^{-ig}.$$
These equations imply that $p_5$ is a non-zero constant, $p/p_2$ is a non-zero polynomial  and furthermore,
$$e^{i(\overline{g}+g)}=\frac{p\overline{p_5}}{p_2}$$ 
It follows that $\overline{g}+g$ is constant, $p/p_2$ is a non-zero constant function and $g$ is transcendental by Lemma \ref{const}. Thus, $g$ has the property (ii) in the theorem.\vskip 1mm
\par
{\it Sub-case 1.4:}\,$\beta_2\equiv\delta_2\equiv 0$, i.e., $b_2\equiv\tilde{b}_2\equiv 0$. We obtain a contradiction to $d_1\not\equiv 0$.
This means that \eqref{E1.20} does not have any entire solutions.
\vskip 1mm
\par
{\it Case 2:}\,exactly one of $\beta_2,$ $\gamma_2$ and $\delta_2$ is equal to zero.
An application of Lemma \ref{new3} yields that $\overline{g}$ is constant, which is the property (i) in the theorem.\vskip 1mm
\par
{\it Case 3:}\,$\alpha_2\beta_2\gamma_2\delta_2\not\equiv 0$. Clearly,
$$
\alpha_2\beta_2-\gamma_2\delta_2=-(d_1)^2\overline{pb_1b_2}+\overline{d_1}^2p\tilde{b}_1\tilde{b}_2.
$$
Applying Lemma \ref{new2} to \eqref{E1.23} with $u=1$ and $m=0$ gives
$$
-(d_1)^2\overline{pb_1b_2}+\overline{d_1}^2p\tilde{b}_1\tilde{b}_2\equiv 0,
$$
and either
$$e^{i(\overline{g}+g)}=\frac{\delta_2}{\alpha_2}\quad \mbox{or} \quad
e^{-i(\overline{g}-g)}\equiv \frac{\delta_2}{\beta_2}.
$$
These imply that  $\overline{g}+g$ and $e^{i(\overline{g}+g)}$ or $\overline{g}-g$ and $e^{i(\overline{g}-g)}$ are constant since $\delta_2/\alpha_2$ and $\delta_2/\beta_2$ are rational functions.
Hence, in this situation, the property (iii) in the theorem happens.
\end{proof}
\vskip.05in
\begin{remark}
\cite[Theorem 2.1]{zhengXu-2022} states that \eqref{E1.20} does not have any transcendental entire solution with
finite order if $p_5=-p_6=p_4=p=1$ and $p_2=0$. Our Theorem \ref{T1.3} strengthens \cite[Theorem 2.1]{zhengXu-2022} and shows \eqref{E1.20} only has constant solutions in this setting since $b_1b_2=1$ and $\tilde{b}_1\tilde{b}_2=2$.
\end{remark}

The following two examples show that both (ii) and (iii) in Theorem \ref{T1.3} could happen.\par

\begin{example}Let $p$ be any nonzero complex number and
$$p_5=p_2=i \quad \mbox{and} \quad p_6=-p_4=-1.$$
Set $g(z_1, z_2)=e^{iz_2}+a$ and $c=(0, c_2)\in \C^2\setminus\{0\}$, where $a\in\mathbb{C}$ and $c_2\in \C$ satisfy $e^{i2a}=p$ and $e^{ic_2}=-1, respectively.$ Then $f(z)=\frac{p}{2i}e^{-ig}$ is an entire solution to the equation \eqref{E1.20} and the property (ii) in Theorem \ref{T1.3} exists.
\end{example}

\begin{proof}
Since $\overline{g}=e^{i(z_2+c_2)}+a$, thus $g+\overline{g}=e^{iz_2}(1+e^{ic_2})+2a=2a$, and we have $e^{i(g+\overline{g})}=e^{i2a}=p.$ Hence,
\begin{eqnarray*}
\left(i\overline{f}+if\right)^2+\left(\overline{f}-f\right)^2=-4f\overline{f}=p^2e^{-i(g+\overline{g})}=p.
\end{eqnarray*}
Clearly, $d_1=-2i$, and
$$ \tilde{b}_1 =(k+1)i, \ \   \tilde{b}_2 =(k-1)i, \ \
b_1  =(k-1)i, \ \ b_2 =(k+1)i.
$$
Thus, $f$ has the form \eqref{E1.3-4} with $k=1$ and is an entire solution to the equation \eqref{E1.20}. Now $\tilde{b}_2=b_1=0$, so the property (ii) in Theorem \ref{T1.3} exists.
\end{proof}

\begin{example}Let  $p$ be any nonzero constant and
$$p_5=p_2=-p_6=p_4=1.$$
Set $g(z_1, z_2)=z_1+h(z_2)$ and $c=(c_1, 0) \in \C^2\setminus\{0\}$ such that $e^{ic_1}=-i,$ where $h(z_2)$ is a transcendental entire function of the variable $z_2.$ Then
$$f(z)=\frac{1}{4}\left((1+i)e^{ig}+p(1-i)e^{-ig}\right)$$
is an entire solution to the equation \eqref{E1.20} and the property (iii) in Theorem \ref{T1.3} exists.
\end{example}

\begin{proof}
Let $k=1$. Then $d_1=-2$,
$$ \tilde{b}_1 =1+i, \quad   \tilde{b}_2 =1-i, \quad
b_1  =1-i  \quad \mbox{and} \quad  b_2 =1+i.
$$
Thus,
$$
f=\frac14\left(1+i \quad p(1-i)\right)
\left(\begin{array}{c} e^{ig} \\ [5pt] e^{-ig}\end{array}
\right)
=\frac1{-4i}\left(b_1 \quad -pb_2\right)
\left(\begin{array}{c} e^{ig} \\ [5pt] e^{-ig}\end{array}
\right).
$$
Since $e^{i\overline{g}}=e^{ig+ic_1}=-ie^{ig}$, we obtain
$$
\overline{f} =\frac14\left(1+i \quad p(1-i)\right)
\left(\begin{array}{c} e^{i\og} \\ [5pt] e^{-i\og}\end{array}
\right)
=\frac1{-4i}\left(-\tilde{b}_1 \quad p\tilde{b}_2\right)
\left(\begin{array}{c} e^{ig} \\ [5pt] e^{-ig}\end{array}
\right).
$$
Hence, $(\overline{f} \quad f)^T$ satisfies the matrix equation \eqref{E1.21}, consequently, $f$
is an entire solution to the equation \eqref{E1.20}.
Clearly,
$$
e^{i(\overline{g}-g)}=e^{ic_1}=-i=-\frac{\overline{d_1}\,\tilde{b}_1}{d_1\overline{b_1}}\quad \mbox{and} \quad
(d_1)^2\overline{pb_1b_2}=\overline{d_1}^2p\tilde{b}_1\tilde{b}_2=8p \not=0.
$$
\end{proof}

The property (ii) in Theorem \ref{T1.3} can be strengthened to be the following theorem.

\begin{theorem}\label{3rdIFF}
Let $f$ be a non-constant entire solution to \eqref{E1.20} as in Theorem \ref{T1.3}, where $g$ is transcendental. Then we have:
\begin{enumerate}
\item[(a)] 
$
f(z)=\frac{1}{2p_5}e^{ig}$ and $e^{i(\overline{g}+g)} =\frac{p\overline{p_5}}{p_2}
$
if and only if $\tilde{b}_1\equiv b_2\equiv 0$, where $p_5$ and $\frac{p}{p_2}$ are non-zero constants.
\vskip.1in
\item[(b)]
$
f(z)=\frac{p}{2p_5}e^{-ig}$ and $e^{i(\overline{g}+g)} =\frac{p_2\overline{p}}{\overline{p_5}}
$
if and only if $\tilde{b}_2\equiv b_1\equiv 0$, where $p_2$ and $\frac{p}{p_5}$ are non-zero constants.
\end{enumerate}
\end{theorem}
\begin{proof} ($\Leftarrow$) in (a) and (b):  It is directly from {\it Case 1} in the proof Theorem \ref{T1.3}.

 ($\Rightarrow$) in (a): Since
$
f(z)=\frac{1}{2p_5}e^{ig},
$
we obtain from  \eqref{E1.3-4} that
$$
\left(\frac{b_1}{2id_1}-\frac{1}{2p_5}\right)e^{ig}-\frac{pb_2}{2id_1}e^{-ig} \equiv 0.
$$
Therefore, $b_2\equiv0$, i.e., $\beta_2\equiv 0$. Thus by $e^{i(\overline{g}+g)}=\frac{p\overline{p_5}}{p_2}$, \eqref{E1.23} turns to be
$$
\alpha_2\frac{p\overline{p_5}}{p_2}=\gamma_2e^{2ig}+\delta_2.
$$
Since $g$ is not constant, so $\gamma_2\equiv 0$, i.e., $\tilde{b}_1\equiv 0$.

($\Rightarrow$) in (b): Similar to (a), by
$
f(z)=\frac{p}{2p_5}e^{-ig}
$
and \eqref{E1.3-4}, we yield
$$
\frac{b_1}{2id_1}e^{ig}-\left(\frac{pb_2}{2id_1}+\frac{p}{2p_5}\right)e^{-ig} \equiv 0.
$$
Since $g$ is non-constant, so $b_1\equiv0$, i.e., $\alpha_2\equiv 0$. Thus by $e^{i(\overline{g}+g)} =\frac{p_2\overline{p}}{\overline{p_5}}$,
\eqref{E1.23} turns to be
$$
\beta_2\frac{\overline{p_5}}{p_2\overline{p}}=\gamma_2+\delta_2e^{-2ig},
$$
which implies that $\delta_2\equiv 0$, i.e., $\tilde{b}_2\equiv 0$.
\end{proof}




\section{Case IV.\, $D\not\equiv 0$ and $d_2\not\equiv 0$}
\par In this section, we consider a special case of $L(f)$ as
 $$L^*(f):=\sum_{j=1}^{s}q_jf_{z_j}, \quad (1\leq s\leq n)$$
 where all $q_{j}$ are non-zero constants in $\mathbb{C}$. Clearly, $L^{*}(\overline{f})=\overline{L^*(f)}.$

\begin{theorem} \label{T6.1}Let $p_1, \ldots, p_6, p$ be defined as in Theorem \ref{T1} and $Dd_2\not\equiv 0$.
Set
$$
 A\stackrel{def}{=}\left(\begin{array}{cc}
   -d_2 &  D  \\ [.2cm]
  -L^*(d_2)-\frac{d_1d_2}{D} &  L^*(D)+\frac{D\overline{d_1}}{\overline{D}} \\
  \end{array}\right) \ \mbox{and} \ |A|\stackrel{def}{=}\det A.
  $$
Assume that $f$ is an  entire solution to
\begin{align}\label{E23.8.1}
(p_1L^*(f)+p_2\of+p_5f)^2+(p_3L^*(f)+p_4\of+p_6 f)^2=p.
\end{align}

If $|A|\not\equiv 0$, then
\begin{equation}\label{E23-8-11}
f = \frac{1}{2i|A|}K\left(
  \begin{array}{c}
   e^{ig} \\[.2cm]  e^{-ig}
  \end{array}
  \right)
+\frac1{2i|A|}F \left(
  \begin{array}{c}
   e^{i\overline{g}} \\[.2cm]  e^{-i\overline{g}}
  \end{array}
\right),
\end{equation}
where $g$ is an entire function,
$$\ds F=
\left( - \frac{\overline{b_1}D^2}{\overline{D}} \qquad  \frac{\overline{pb_2}D^2}{\overline{D}} \right),
$$
 and
$K=(k_{11}\quad k_{12})$ is a $1\times 2$ matrix with elements
\begin{equation*}
\begin{split}
k_{11}&=a_1(L^*(D)+\frac{D\overline{d_1}}{\overline{D}})+b_1d_2-D(L^*(a_1)+ia_1L^*(g)) \ \mbox{and} \\
k_{12}&=-pa_2(L^*(D)+\frac{D\overline{d_1}}{\overline{D}})-pb_2d_2+D(L^*(pa_2)-ipa_2L^*(g)).
\end{split}
\end{equation*}
Further, if $b_1b_2\not\equiv 0$, then $g$ has one of the following properties:
\begin{enumerate}
\item[(i)] $g$ is constant.
\item[(ii)] $L^*(\og)\equiv 0$.
\item[(iii)] $L^*(\og)$ is a polynomial, and either $g+\overline{g}$ or $g-\overline{g}$ is constant.
\item[(iv)] $L^*(g)+L^*(\og)$ and $L^*(g)-L^*(\og)$ are transcendental.
\end{enumerate}
If $b_1b_2\equiv 0$, then $g$ and $\ug$ satisfy a non-linear partial differential equation with degree $4$.

If $|A|\equiv 0$, then we have
$$-2id_2f+2iD\of=a_1e^{ig}-pa_2e^{-ig}$$
and
$$2iDL^*(f)-2id_1f=-b_1e^{ig}+pb_2e^{-ig},$$
where $g$ is an entire function having  one of the following properties:
\begin{enumerate}
\item[(a)] $g$ is constant.
\item[(b)] $L^*(g)\equiv 0$ and either $g+\overline{g}$ or $g-\overline{g}$ is constant.
\item[(c)] $L^*(g)$ is a polynomial and $g-\overline{g}$ is constant
\end{enumerate}
\end{theorem}

\begin{proof} We apply Lemma \ref{equiv} with $s(f)=L^*(f),\ t(f)=\overline{f}$ and $u(f)=f$ to \eqref{E23.8.1} and get
\begin{align*}
\left(\begin{array}{c}
    L^*(f)  \\ [5pt] \overline{f}
\end{array}\right)
= \frac{1}{2iD}\left(
  \begin{array}{cc}
    -b_1 & pb_2 \\ [5pt] a_1 & -pa_2
  \end{array}
\right)E
+
 \frac{1}{D} \left(
  \begin{array}{c}
  d_1 \\ [5pt] d_2
  \end{array}
\right)f,
\end{align*}
where $E=\left(
     e^{ig} \quad e^{-ig}\right)^T$ is a $2\times 1$ matrix, and which gives
\begin{equation}\label{23.8.2}
DL^*(f)=\frac{1}{2i}\left(
   -b_1\quad  pb_2
  \right)E+d_1f
\end{equation}
and
\begin{eqnarray}\label{23.8.3}
\left(
   -d_2\quad  D
  \right)\left(
  \begin{array}{c}
   f  \\ [.1cm] \overline{f}
  \end{array}
\right)
&=&
\frac{1}{2i}\left(
   a_1\ \  -pa_2
  \right)\left(
  \begin{array}{c}
   e^{ig} \\[.1cm]  e^{-ig}
  \end{array}
\right).
\end{eqnarray}
Taking partial differentiation on both sides of \eqref{23.8.3} gives
\begin{align}\label{23.8.4}
&& \left(
   -L^*(d_2)\quad  L^*(D)
  \right)\left(
  \begin{array}{c}
   f  \\ [.2cm] \overline{f}
  \end{array}
\right)
+
\left(
 -d_2\quad  D
  \right)\left(
  \begin{array}{c}
   L^*(f)  \\ [.2cm] L^*(\overline{f})
  \end{array}
\right) \\ \nonumber
&&=
\frac{1}{2i}\big(
  L^*(a_1)\ \  -L^*(pa_2)
  \big) E+
  \frac{1}{2i}\big(
ia_1L^*(g)\ \ ipa_2L^*(g)
  \big)  E.
\end{align}
Combing \eqref{23.8.3} with \eqref{23.8.4} gives
\begin{align}\label{23.8.5}
&  \left(\begin{array}{cc}
   -d_2 &  D  \\ [.2cm]
 -L^* (d_2) &  L^*(D) \\
  \end{array}\right)
\left(
  \begin{array}{c}
   f  \\ [.2cm] \overline{f}
  \end{array}
\right)
=
\left(\begin{array}{cc}
    0 & 0 \\[.2cm]
   d_2 &  -D\\
  \end{array}\right)\left(
  \begin{array}{c}
   L^*(f)  \\ [.2cm] L^*(\overline{f})
  \end{array}
\right)
 \\ \nonumber
& \qquad \qquad  +
\frac{1}{2i}\left(\begin{array}{cc}
 a_1&  -pa_2\\[.2cm]
  L^*(a_1)+ia_1L^*(g)  & -L^*(pa_2)+ipa_2L^*(g)\\
  \end{array}\right)
  E. \\ \nonumber
\end{align}
 Set $\overline{E}=\left(
     e^{i\overline{g}} \quad e^{-i\overline{g}}\right)^T$, which is the shift of $E$. Then by \eqref{23.8.2}, we get
 \begin{align}\label{23.8.6}
 & \left(
  \begin{array}{c}
   L^*(f)  \\ [.2cm] L^*(\overline{f})
  \end{array}
\right)
=\frac{1}{2i}\left(\begin{array}{cc}
    \frac{-b_1}{D} & \frac{pb_2}{D} \\[.2cm]
   0 &  0\\
  \end{array}\right)
 E\\ \nonumber
&\qquad
+
\frac{1}{2i}\left(\begin{array}{cc}
0 & 0 \\ [.2cm]
    \frac{-\overline{b_1}}{\overline{D}} & \frac{\overline{pb_2}}{\overline{D}}
  \end{array}\right)
   \overline{E}+\left(\begin{array}{cc}
\frac{d_1}{D} & 0 \\ [.2cm]
    0 & \frac{\overline{d_1}}{\overline{D}}
  \end{array}\right)
   \left(
  \begin{array}{c}
   f \\[.2cm]  \overline{f}
  \end{array}\right).
\end{align}
Plugging \eqref{23.8.6} into \eqref{23.8.5} gives
\begin{equation}\label{23.8.7}
A\left(
  \begin{array}{c}
   f  \\ [.2cm] \overline{f}
  \end{array}
\right)
=\frac{1}{2i}B E
+
\frac1{2i}C
  \overline{E},
\end{equation}
where the matrices $A=(a_{ij})_{2\times 2}, B=(b_{ij})_{2\times 2}, C=(c_{ij})_{2\times 2}$ are defined by
\begin{align*}
A = \left(\begin{array}{cc}
   -d_2 &  D  \\ [.2cm]
  -L^*(d_2)-\frac{d_1d_2}{D} &  L^*(D)+\frac{D\overline{d_1}}{\overline{D}} \\
  \end{array}\right),\quad  C=\left(\begin{array}{cc}
0 & 0 \\ [.2cm]
    \frac{\overline{b_1}D}{\overline{D}} & -\frac{\overline{pb_2}D}{\overline{D}}
  \end{array}\right),\\
  B=\left(\begin{array}{cc}
  a_1&  -pa_2\\[.2cm]
   L^*(a_1)-\frac{b_1d_2}{D}+ia_1L^*(g)  &\,\, -L^*(pa_2)+\frac{pb_2d_2}{D}+ipa_2L^*(g)\\
  \end{array}\right).
 \end{align*}
\vskip 2mm
{\it Case 1:} $A$ is invertible. Then $$|A|=-d_2\Big(L^*(D)+\frac{D\overline{d_1}}{\overline{D}}\Big)+DL^*(d_2)+d_1d_2\not\equiv 0$$  and
 $$
 A^{-1}
 =\frac1{|A|}\left(\begin{array}{cc}
    L^*(D)+\frac{D\overline{d_1}}{\overline{D}} & -D   \\ [.2cm]
  L^*(d_2)+\frac{d_1d_2}{D}  &   -d_2
  \end{array}\right)
 :=\frac1{|A|}\left(\begin{array}{cc}
    a_{22} &  -a_{12}  \\ [.2cm]
  -a_{21}  &   a_{11}
  \end{array}\right).
  $$
Noting $c_{11}=c_{12}=0$, we yield  from \eqref{23.8.7} that
\begin{equation*}
\left(
  \begin{array}{c}
   f  \\ [.2cm] \of
  \end{array}
\right)
=\frac{1}{2i}A^{-1}B E
+\frac1{2i}A^{-1}C \overline{E},
\end{equation*}
which gives
\begin{equation}\label{E23-8-11.1}
\begin{split}
2i|A|f =& \big(a_{22}b_{11}-a_{12}b_{21}\quad a_{22}b_{12}-a_{12}b_{22}\big) E\\
&+\big(-a_{12}c_{21}\quad -a_{12}c_{22}\big)\overline{E}.
\end{split}
\end{equation}
By a straight computation of \eqref{E23-8-11.1}, we obtain \eqref{E23-8-11}  in our theorem. Taking partial derivative on both sides of \eqref{E23-8-11} gives

\begin{align*}
L^*(f)& =
\left\{L^*\left(\frac{K}{2i|A|}\right)+\frac{K}{2i|A|}
  \left(\begin{array}{cc}
    iL^*(g) &  0  \\ [.2cm]
   0 &   -iL^*(g)
  \end{array}\right)
   \right\} E
  \\ \nonumber
  & \qquad
  +  \left\{L^*\left(\frac{F}{2i|A|}\right)+\frac{F}{2i|A|}
   \left(\begin{array}{cc}
    iL^*(\overline{g}) &  0  \\ [.2cm]
   0 &   -iL^*(\overline{g})
  \end{array}\right)
  \right\}\overline{E}.
\end{align*}
Combining this and \eqref{E23-8-11} with \eqref{23.8.2} yields
\begin{equation}\label{8.07B}
(\alpha \quad  \beta)
 \left(
  \begin{array}{c}
   e^{i\overline{g}} \\[.2cm]  e^{-i\overline{g}}
  \end{array}
\right) =
(\gamma \quad \delta)\left(
  \begin{array}{c}
   e^{i{g}} \\[.2cm]  e^{-i{g}}
  \end{array}
\right),
\end{equation}
where
\begin{align}\label{24-04-20}
(\alpha \quad  \beta)
&=D\left\{L^*\left(\frac{F}{|A|}\right)+\frac{F}{|A|}
   \left(\begin{array}{cc}
    iL^*(\overline{g}) &  0  \\ [.2cm]
   0 &   -iL^*(\overline{g})
  \end{array}\right)
  \right\}-\frac{d_1}{|A|}F,
  \\ \nonumber
 \alpha &=-DL^*\left(\frac{\overline{b_1}D^2}{|A|\overline{D}}\right)-i\frac{\overline{b_1}D^3}{|A|\overline{D}}L^*(\overline{g})+\frac{d_1\overline{b_1}D^2}{|A|\overline{D}},
 \\ \nonumber
 \beta &=DL^*\left(\frac{\overline{pb_2}D^2}{|A|\overline{D}}\right)-i\frac{\overline{pb_2}D^3}{|A|\overline{D}}L^*(\overline{g})-\frac{d_1\overline{pb_2}D^2}{|A|\overline{D}};
\end{align}
and
\begin{align*}
(\gamma \ \  \delta)
& =(-b_1\quad pb_2)+\frac{d_1}{|A|}K
-D\left\{L^*\left(\frac{K}{|A|}\right)+\frac{K}{|A|}
  \left(\begin{array}{cc}
    iL^*(g) &  0  \\ [.2cm]
   0 &   -iL^*(g)
  \end{array}\right)
   \right\}, 
  \\ \nonumber
\gamma &=-b_1+\frac{d_1x}{|A|}-DL^*\left(\frac{x}{|A|}\right)+i\left(\frac{x}{|A|}-L^*\left(\frac{a_1D}{|A|}\right)-\frac{d_1a_1}{|A|}\right)DL^*(g)\\ \nonumber
& \qquad +i\frac{a_1D^2}{|A|}L^*(L^*(g))+\frac{a_1D^2}{|A|}\left(L^*(g)\right)^2,
\\ \nonumber
\delta & =pb_2+\frac{d_1y}{|A|}-DL^*\left(\frac{y}{|A|}\right)+i\left(L^*\left(\frac{pa_2D}{|A|}\right)+\frac{y}{|A|}-\frac{d_1pa_2}{|A|}\right)DL^*(g)\\ \nonumber
& \qquad +i\frac{pa_2D^2}{|A|}L^*(L^*(g))+\frac{pa_2D^2}{|A|}\left(L^*(g)\right)^2.
\end{align*}
Here, $x,y$ are the terms not involving $L^*(g)$ and defined by
\begin{equation*}
\begin{split}
x&=a_1(L^*(D)+\frac{D\overline{d_1}}{\overline{D}})+b_1d_2-DL^*(a_1),  \\
y&=-pa_2(L^*(D)+\frac{D\overline{d_1}}{\overline{D}})-pb_2d_2+DL^*(pa_2).
\end{split}
\end{equation*}

To obtain the properties of $g$, we consider all possible cases on coefficients $\alpha, \beta, \gamma, \delta$ in \eqref{8.07B} under the assumption that
$g$ is non-constant. Otherwise, $g$ has the property (i).

{\it Sub-case 1.1:} none of $\alpha, \beta, \gamma$ and $\delta$ is identically equals to zero.
Set
$$
h_1=L^*(g)+L^*(\overline{g}) \quad \mbox{and} \quad h_2=L^*(g)-L^*(\overline{g}).
$$
\par First, when  $h_1$  is a polynomial, we make $\alpha$ and $\beta$ to be polynomials in $L^*(g)$ with rational coefficients involving $h_1$.
Thus, set $h=L^*(g),$ noting $\gamma$ and $\delta$ are polynomials in $h, h^2$ and
$L^*(h)$ with rational coefficients and
$$
T(r, h)\le O(T(r, g)) \quad \mbox{and} \quad T(r, L^*(h))\le O(T(r, h))\le O(T(r, g)),
$$
we get $\alpha\beta-\gamma\delta\equiv 0$ by a similar argument in the proof of the first part of Lemma \ref{new2},
where
$$
\alpha\beta-\gamma\delta
=u_1h^4+u_2h^3+u_3h^2L^*(h)+u_4h^2+u_5hL^*(h)+u_6h+u_7L^*(h)+u_8,
$$
where $u_1=-|A|^{-2}D^2pa_1a_2$ and  $u_j$ ($j=2,\cdots, 8$) are rational functions.\par

Assume $u_1\not\equiv 0.$ Since $\alpha\beta-\gamma\delta=0$,
\begin{equation}\label{sec8newA}
h=v_2+v_3\frac{L^*(h)}{h}+\frac{v_4}{h}+\frac{v_5}h\frac{L^*(h)}{h}+\frac{v_6}{h^2}+\frac{v_7}{h^2}\frac{L^*(h)}{h}+\frac{v_8}{h^3},
\end{equation}
where $v_j=u_j/u_1$ ($j=2,\cdots, 8$) are rational coefficients. It follows that
\begin{align*}
T(r, h)
&=
\int_{S_{n}(r)}\log^{+}|h(z)|\sigma_{n}(z)=\int_{S_{n}(r)\cap \{z: |h(z)|\ge 1\}}\log^{+}|h(z)|\sigma_{n}(z)\\
&\le
3m(r, \frac{L^*(h)}{h})+ O(\log r)\le 3\sum_{j=1}^{s}m(r, \frac{h_{z_{j}}}{h})+ O(\log r)\\
&=O(\log T(r, h))+O(\log r),
\end{align*}
for $r \not\in E \subset [1,\infty)$, where $E$ is a set of finite Lebesgue measure. Here, we also use the lemma of logarithmic derivative in the above estimation. Thus, $h$ is a polynomial.\par

Assume $u_1\equiv 0,$ e.g., either $a_1\equiv 0$ or $a_2\equiv 0.$ Clearly, $a_1\equiv 0$ and $a_2\equiv 0$ can not happen at the same time since $D\not\equiv 0$.  Whenever $a_1\equiv 0$, $b_1\not\equiv 0$ by $D\not\equiv 0$, so we have $x=b_1d_2\not\equiv 0$, and
$$\gamma =-b_1+\frac{d_1x}{|A|}-DL^*\left(\frac{x}{|A|}\right)+i\frac{x}{|A|}L^*(g).$$
Then we have
$$0\equiv\alpha\beta-\gamma\delta=u_1^{*}h^3+u_2^{*}h^2+u_3^{*}hL^*(h)+u_4^{*}h+u_5^*L^*(h)+u_6^{*},$$
where $u_1^{*}:=-i|A|^{-2}D^3pa_2x\not\equiv 0$ and  $u_j^*$ ($j=2, 3, 4, 5, 6$) are rational functions, which leads
\begin{equation}h=v_2^{*}+v_3^{*}\frac{L^*(h)}{h}+\frac{v_4^{*}}{h}+\frac{L^*(L^*)}{h}\frac{v_5^{*}}{h}+\frac{v_6^{*}}{h^2},
\end{equation} where $v_j^{*}=u_j^{*}/u_1^{*}$ ($j=2, 3, 4, 5, 6$) are rational coefficients. Similarly as in the argument for $u_1\not\equiv 0$, it follows that
\begin{align*}
T(r, h)
&\le
2\sum_{j=1}^{s}m(r, \frac{h_{z_j}}{h})+ O(\log r)\\&=O(\log T(r, h))+O(\log r),\quad r\not\in E
\end{align*}
where $E$ is a set of finite Lebesgue measure. Hence, $h$ is a polynomial. \par

Whenever $a_2\equiv 0$, then $a_1\not\equiv 0$ and $b_2\not\equiv 0$ since $D\not\equiv 0$. At the same time, $y=-pb_2d_2\not\equiv 0$ and
$$\delta =pb_2+\frac{d_1y}{|A|}-DL^*\left(\frac{y}{|A|}\right)-i\frac{y}{|A|}L^*(g).$$
By using similar discussion as the case of $a_1\equiv 0,$ we also obtain that $h$ is a polynomial. \par

So, we get that $L^*(g)$ is a polynomial. By a similar argument used in the proof of the second part of Lemma \ref{new2}, we obtain that either
$g+\overline{g}$ or $g-\overline{g}$ is constant, that is, the property (iii).\par

Secondly, when $h_2$ is a polynomial, in the same manner, we also get that $L^*(g)$ is a polynomial, and either $g+\overline{g}$ or $g-\overline{g}$ is constant.\par\par

Lastly, when both $h_1$ and $h_2$ are transcendental, we obtain that $g$ has the property (iv).\par

{\it Sub-case 1.2:} exactly one of $\alpha, \beta, \gamma$ and $\delta$ is identically equals to zero.
We claim that $\alpha, \beta, \gamma$ and $\delta$ are rational functions in this sub-case.\vskip 1mm

{\it \ \ Sub-case 1.2.1:} $b_1b_2\not\equiv 0$. If $\alpha\equiv 0$ or $\beta\equiv 0$, then $L^*(\overline{g})$ is a rational function, that is, $L^*(g)$ is a rational function. Thus, we get that $\alpha, \beta, \gamma$, and $\delta$ are rational functions.

When $\gamma\equiv 0$, we have
\begin{eqnarray*}\label{gamma=0}&&0\equiv -b_1+\frac{d_1x}{|A|}-DL^*\left(\frac{x}{|A|}\right)+i\left(\frac{x}{|A|}+L^*\left(\frac{a_1D}{|A|}\right)-\frac{d_1a_1}{|A|}\right)DL^*(g)\\ \nonumber
&& \qquad +i\frac{a_1D^2}{|A|}L^*(L^*(g))+\frac{a_1D^2}{|A|}\left(L^*(g)\right)^2,\end{eqnarray*}
If $a_1=kp_1-ip_3\not\equiv 0,$ then by the lemma of logarithmic derivative and using a similar discussion as in Sub-case 1.1 we get $$T(r, L^*(g))=O(\log r)$$ for $r \not\in E \subset [1,\infty),$  where $E$ is a set of finite Lebesgue measure.  If $a_1\equiv 0,$ then $b_1\not\equiv 0$ by $D\not\equiv 0$, and\begin{eqnarray*}\label{gamma=0}&&0\equiv -b_1+\frac{b_1d_1d_2}{|A|}-DL^*\left(\frac{b_1d_2}{|A|}\right)+i\frac{b_1d_2}{|A|}DL^*(g).\end{eqnarray*} This implies $T(r, L^*(g))=O(\log r),$ too. Thus $L^*(g)$ is a polynomial, and $\alpha, \beta$ and $\delta$ are rational functions.\par

Similarly,  when $\delta\equiv 0,$ we have \begin{eqnarray*}\label{delta=0}&&0\equiv pb_2+\frac{d_1y}{|A|}-DL^*\left(\frac{y}{|A|}\right)+i\left(L^*\left(\frac{pa_2D}{|A|}\right)+\frac{y}{|A|}-\frac{d_1pa_2}{|A|}\right)DL^*(g)\\ \nonumber
&& \quad +i\frac{pa_2D^2}{|A|}L^*(L^*(g))+\frac{pa_2D^2}{|A|}\left(L^*(g)\right)^2.\end{eqnarray*}
If $a_2=kp_1+ip_3\not\equiv 0,$ then by the lemma of logarithmic derivative and using a similar discussion as in Sub-case 1.1 we get $T(r, L^*(g))=O(\log r)$ for $r \not\in E \subset [1,\infty),$  where $E$ is a set of finite Lebesgue measure. If $a_2\equiv 0,$ then we have $b_2\not\equiv 0$ by $D\not\equiv 0$, and
\begin{eqnarray*}&&0\equiv pb_2-\frac{pb_2d_1d_2}{|A|}+DL^*\left(\frac{pb_2d_2}{|A|}\right)-i\frac{pb_2d_2}{|A|}DL^*(g).\end{eqnarray*}
This again implies $T(r, L^*(g))=O(\log r)$. Thus, we get $\alpha, \gamma$ and $\delta$ are rational functions.\par

Therefore, the claim is proved. Applying Lemma \ref{new3} to this sub-case, we get that the property (i) holds.\par

{\it \ \ Sub-case 1.2.2:} $b_1b_2\equiv 0.$ Since $D\not\equiv 0,$ $b_1$ and $b_2$ cannot be identically equal to zero at the same time. Without loss of generality, we assume $b_1\not\equiv 0$ and $b_2\equiv 0$.
Thus, \eqref{E23-8-11.1} becomes
\begin{equation}\label{bad1}
i|A|f = \big(a_{22}b_{11}-a_{12}b_{21}\quad a_{22}b_{12}-a_{12}b_{22}\big) E
-a_{12}c_{21}e^{i\og}.
\end{equation}
We take the operator $L^*$ on the both sides of the equation and obtain
\begin{equation}\label{bad2}
L^*\Big(\frac{i|A|f - \big(a_{22}b_{11}-a_{12}b_{21}\quad a_{22}b_{12}-a_{12}b_{22}\big) E}{a_{12}c_{21}}\Big)
=iL^*(\og))e^{i\og}.
\end{equation}
It follows from \eqref{bad1} and \eqref{bad2} that we obtain an equation without the factor $e^{i\og}$.
Using the same technique to the equation, we can eliminate $e^{ig}$ and $e^{-ig}$. Hence, we yield that
$g$ and $\og$ satisfy a non-linear partial differential equation of degree $4$.

Therefore, {\it Sub-case 1.2} is completed.

{\it Sub-case 1.3:} exactly two of $\alpha, \beta, \gamma$ and $\delta$ are identically equal to zero.

Firstly, when $\gamma\equiv\delta\equiv 0$ and $\alpha\beta \not\equiv 0,$  we
we can write \eqref{8.07B} as
\begin{equation*}
{\alpha} e^{i2\overline{g}}+\beta=0.
\end{equation*}
Since $D\not\equiv 0$, $a_1\equiv 0$ and $a_2\equiv 0$ can not hold at the same time. It follows from $\gamma=\delta\equiv 0$ that $L^*(g)$ is a polynomial. Thus, $\alpha$ and $\beta$ are rational functions.
Therefore, $g$ is constant, i.e., $g$ has the property (i).\par

Secondly, when $\alpha\equiv \beta \equiv 0$ and $\gamma\delta\not\equiv 0$, we
we can write \eqref{8.07B} as
\begin{equation*}
{\gamma} e^{i2g}+\delta=0.
\end{equation*}
Since $D\not\equiv 0,$ at least one of $b_1$ and $b_2$  is not identically equal to zero. Then it follows from $\alpha\equiv \beta\equiv 0$ that $L^*(g)$ is a polynomial. Thus, $\gamma$ and $\delta$ are rational functions.
Therefore, $g$ is constant, i.e., $g$ has the property (i).

Thirdly, when $\alpha\equiv \gamma \equiv 0$ and $\beta\delta\not\equiv 0$, we get
$$
\beta=\delta e^{-i(g-\og)}.
$$
Similarly as in Sub-case 1.2, $\gamma \equiv 0$ implies that $L^*(g)$ is polynomial. Then both $\beta$ and $\delta$ are rational functions, and thus $g-\og$ is constant.  Hence, $g$ has the property (iii).

Fourthly, when $\alpha\equiv \delta \equiv 0$ and $\beta\gamma\not\equiv 0$, we have
$$
\beta=\gamma e^{i(g+\og)}.
$$
Similar to the third situation, $L^*(g)$ is polynomial and $g+\og$ is constant. Thus, $g$ has the property (iii).

In a similar manner, when the last situations $\beta\equiv \gamma\equiv 0$ and $\beta\equiv \delta\equiv 0$ occur, we also get $g$ has either the property (i) or the property (iii).


{\it Sub-case 1.4:} exactly three of $\alpha, \beta, \gamma$ and $\delta$ are identically equal to zero. Obviously, from \eqref{8.07B}, $\alpha\equiv \beta\equiv \gamma\equiv \delta\equiv 0$. Since $b_1,b_2$ can not be identical to zero at the same time, $\alpha\equiv\beta\equiv 0$ leads immediately that $L^*(\overline{g})$ is a polynomial. Further  from the exact forms of $\alpha,\beta$ in \eqref{24-04-20}, if $b_1b_2\not\equiv 0$, we get
\begin{equation*}
0=-\overline{pb_2}L^*\left(\frac{\overline{b_1}D^2}{|A|\overline{D}}\right)+\overline{b_1}L^*\left(\frac{\overline{pb_2}D^2}{|A|\overline{D}}\right)-2i\frac{\overline{pb_1b_2}D^2}{|A|\overline{D}}L^*(\overline{g}).
\end{equation*}
This implies
$L^*(\overline{g})\equiv 0$ by Liouville's theorem that is just property (ii). When $b_1b_2\equiv 0$, then we can do the same as in {\it Sub-case 1.2.2}.
\vskip 2mm

\par {\it Case 2:} $A=(a_{ij})_{2\times 2}$ is not invertible. This means $a_{11}a_{22}-a_{12}a_{21}=0$, and we set
\begin{equation*}
N=\left(\begin{array}{cc}
    a_{22} &  a_{12}  \\ [.2cm]
  -a_{22}  &   a_{12}
  \end{array}\right).
\end{equation*}
It follows from \eqref{23.8.7} that

\begin{equation}\label{23-8-12}
2a_{22}\left(\begin{array}{cc}
    a_{11} &  a_{12}  \\ [.2cm]
        0  &   0
  \end{array}\right)\left(
  \begin{array}{c}
   f  \\ [.2cm] \of
  \end{array}
\right)=NA\left(
  \begin{array}{c}
   f  \\ [.2cm] \of
  \end{array}
\right)
=\frac{1}{2i} NB E
+\frac{1}{2i} NC\overline{E},
\end{equation}
which gives
\begin{equation*}
(-a_{22}\quad a_{12}) BE+(-a_{22}\quad a_{12})C\overline{E}=0.
\end{equation*}
Then we get
\begin{equation}\label{23-8-12.1}
\alpha_{*} e^{i\overline{g}}+\beta_{*} e^{-i\overline{g}}+\gamma_{*} e^{ig}+\delta_{*} e^{-ig}=0,
\end{equation}
where \begin{eqnarray*}
\alpha_{*} & =&a_{12}c_{21}=\frac{\overline{b_1}D^2}{\overline{D}}, \quad \beta_{*}=a_{12}c_{22}=-\frac{\overline{pb_2}D^2}{\overline{D}}, \\
\gamma_{*} & =&-a_{22}b_{11}+a_{12}b_{21}\\
&=& -a_1(L^*(D)+\frac{D\overline{d_1}}{\overline{D}})-b_1d_2+D(L^*(a_1)+ia_1L^*(g)),\\
\delta_{*} & =&-a_{22}b_{12}+a_{12}b_{22}\\
&=& pa_2(L^*(D)+\frac{D\overline{d_1}}{\overline{D}})+pb_2d_2-D(L^*(pa_2)-ipa_2L^*(g)).
\end{eqnarray*}
The assumption $D\not\equiv 0$ gives that $b_1$ and $b_2$ can not be identically equal to zero at the same time. Thus $\alpha_{*}$ and $\beta_{*}$ can not be identically equal to zero at the same time.  If there exist three of $\alpha_*, \beta_*, \gamma_*$ and $\delta_*$ are identically equal to zero, then all $\alpha_*, \beta_*, \gamma_*$ and $\delta_*$ are identically equal to zero. This is impossible. Below we prove the theorem by considering three cases.\par

{\it Case 2.1:} exactly two of  $\alpha_*, \beta_*, \gamma_*$ and $\delta_*$ are identically equal to zero.\par

{\it Sub-case 2.1.1:} $\alpha_*\equiv \gamma_*\equiv 0.$ Then $b_1\equiv 0$, and $a_1\not\equiv 0$ by $D\not\equiv 0$, further $\gamma_*\equiv 0$ implies $L^*(g)$ is a polynomial. The equation \eqref{23-8-12.1} is reduced to
$$e^{-i(g-\overline{g})}\equiv-\frac{\beta_*}{\delta_*},$$
which is impossible unless  $g-\overline{g}$ is constant. Thus, $g$ has the property (b) in the theorem. \par

{\it Sub-case 2.1.2:} $\alpha_*\equiv \delta_*\equiv 0.$ Clearly $b_1\equiv 0,$ so $b_2\not\equiv 0$, then $a_2\not\equiv 0$ by $D\not\equiv 0$ and $L^*(g)$ is a polynomial. The equation \eqref{23-8-12.1} becomes
$$e^{i(g+\overline{g})}\equiv-\frac{\beta_*}{\gamma_*},$$
which is a contradiction unless $g+\overline{g}$ is constant. This means that $$L^*(g+\overline{g})=L^*(g)+L^*(\overline{g})=0.$$
By Lemma \ref{const}, we have $L^*(g)=0$. Thus, $g$ has the property (b) in the theorem. \par

{\it Sub-case 2.1.3:} $\beta_*\equiv \gamma_*\equiv 0.$ Then $b_2\equiv 0,$ which means $b_1\not\equiv 0$ and $a_2\not\equiv 0$ since $D\not\equiv 0.$ If $a_1\equiv 0$, $\gamma_*\equiv 0$ leads $d_2\equiv 0$, a contradiction. Thus, $\gamma_*\equiv 0$ must imply $L^*(g)$ is a polynomial. Equation \eqref{23-8-12.1} is reduced to
$$e^{i(g+\overline{g})}\equiv-\frac{\delta_*}{\alpha_*},$$
which implies that $g+\overline{g}$ is constant, similarly $L^*(g)=0$ by Lemma \ref{const}. Thus, $g$ has the property (b) in the theorem. \par

{\it Sub-case 2.1.4:} $\beta_*\equiv \delta_*\equiv 0.$ Similarly $b_2\equiv 0$ and $a_2\not\equiv 0$, so $\delta_*\equiv 0$ implies $L^*(g)$ is a polynomial. Equation \eqref{23-8-12.1} becomes
$$e^{i(\overline{g}-g)}\equiv-\frac{\gamma_*}{\alpha_*},$$
then $g-\overline{g}$ is constant, and $g$ has the property (b) in the theorem. \par

{\it Sub-case 2.1.5:}  $\gamma_*\equiv\delta_*\equiv 0.$ Since $D\not\equiv0,$ we get that $a_1$ and $a_2$ can not be identically equal to zero at the same time. Then $L^*(g)$ is a polynomial from either $\gamma_*\equiv 0$ or $\delta_*\equiv 0.$ From \eqref{23-8-12.1}, we have
$$e^{2i\overline{g}}\equiv-\frac{\beta_*}{\alpha_*},$$
which means that $\overline{g}$ is constant. Thus, $g$ has the property (a) in the theorem. \par

{\it Case 2.2:} exactly one of $\alpha_*, \beta_*, \gamma_*$ and $\delta_*$ is identically equal to zero. Applying Lemma 2.9 with $u=L^*(g)$ to \eqref{23-8-12.1}, it always follows that $g$ is constant. Thus $g$ has the property (a) in the theorem.
\par

{\it Case 2.3:} $\alpha_*\beta_*\gamma_*\delta_* \not\equiv 0.$
It is clear that
$$\alpha_*\beta_*-\delta_*\gamma_*
=pa_1a_2D^2(L^*(g))^2+s_1L^*(g)+s_0,
$$ where $s_1$ and $s_0$ are rational functions. Thus, we can apply Lemma \ref{new2} to  \eqref{23-8-12.1} with $u=L^*(g)$ and obtain that
 $\alpha_*\beta_*-\delta_{*}\gamma_{*}\equiv 0,$ and either
$$e^{i(g-\overline{g})}\equiv-\frac{\alpha_{*}}{\gamma_{*}}\equiv-\frac{\delta_{*}}{\beta_{*}} \quad \mbox{or}\quad e^{i(g+\overline{g})}\equiv -\frac{\beta_{*}}{\gamma_{*}}\equiv-\frac{\delta_{*}}{\alpha_{*}}.$$

{\it Sub-case 2.3.1:} whenever $a_1a_2=(kp_1-ip_3)(kp_1+ip_3)\not\equiv 0.$ Then by Lemma \ref{new2} we obtain further that $L^*(g)$ is a polynomial and either $g+\overline{g}$ or $g-\overline{g}$ is constant.

{\it Sub-case 2.3.2:} whenever $a_1\equiv 0$ and $a_2\not\equiv 0.$ Since $D\not\equiv 0,$ we have $b_1\not\equiv 0,$
$$\gamma_{*}=-b_1d_2\not\equiv 0, \quad s_1=ib_1d_2pa_2D\not\equiv 0.$$
Hence by Lemma \ref{new2} we obtain further that $L^*(g)$ is a polynomial and either $g+\overline{g}$ or $g-\overline{g}$ is constant.

{\it Sub-case 2.3.3:}  whenever $a_2\equiv 0$ and $a_1\not\equiv 0.$ We have $b_2\not\equiv 0$ by $D\not\equiv0$, and
$$\delta_{*}=pd_2b_2\not\equiv 0,\quad s_1=-id_2b_2pa_1D\not\equiv 0.$$
Hence by Lemma \ref{new2} we obtain further that $L^*(g)$ is a polynomial and either $g+\overline{g}$ and ${\delta}/{\alpha}$ are constant or  $g-\overline{g}$ and ${\delta}/{\beta}$ are constant.\par

From the above discussion for Sub-cases 2.3.1-2.3.3 we obtain the following conclusions: if $g$ is constant, then $g$ has the property (a); when $L^*(g)$ is polynomial and $g+\overline{g}$ is a constant, then we have $L^*(g)=0$ by Lemma \ref{const}, which means that $g$ has the property (b); when $L^*(g)$ is polynomial and $g-\overline{g}$ is a constant, then $g$ has the property (c).

 \end{proof}

\begin{remark}
It is straightforward to verify that  in Case I,  $|A|\equiv 0$, while in Cases II and III, $D\equiv 0$. Therefore, the method in Theorem \ref{T6.1} cannot be applied to Cases I-III.
\end{remark}


\begin{thebibliography}{99}




\bibitem{ahamed-Allu2023} M. B. Ahamed and V. Allu, Transcendental solutions of Fermat-type functional equations in $\C^n$, Analysis and Mathematical Physics,
13 (2023), no 5, 69-86.

\bibitem{cao-xu-2020} T. B. Cao and L. Xu, Logarithmic difference lemma in several complex variables and partial difference equations, Ann. Mat. Pur. Appl. 199, No. 2, 767-794 (2020).


\bibitem{chang-li-2012} D. C. Chang and  B. Q. Li, Description of entire Solutions of Eiconal type equations, Canad. Math. Bull. 55, no. 2,249-259(2012).

\bibitem{chang-li-yang-1995} D. C. Chang, B. Q. Li and C. C. Yang, On composition of meromorphic functions in
several complex variables, Forum Math. 7, 77-94(1995).




\bibitem{griffiths-book-1976} P. A. Griffiths, Entire holomorphic mappings in one and several complex variables, Princeton University Press and University of Tokyo Press, Princeton, New Jersey, 1976.






\bibitem{han-lv-2019} Q. Han and F. L\"{u}, On the Equation $f^n(z)+g^n(z)=e^{\alpha z+\beta},$ Journal of Contemporary Mathematical Analysis (Armenian Academy of Sciences), 54, No. 2, 98-102(2019).

\bibitem{hemmati1997} J. Hemmati, Entire solutions of first-order nonlinear partial differential equations, Vol. 125, No. 5, May 1997, 1483-1485.

\bibitem{hu-li-yang-book-2003}P. C. Hu, P. Li and C. C. Yang, Unicity of Meromorphic Mappings, Advances in Complex Analysis and its Applications, Volume 1, Kluwer Academic Publishers, Dordrecht/Boston/London, 2003.


\bibitem{huWangWu-2022} P. C. Hu and Q. Wang, On meromorphic solutions of functional equations of Fermat type, Bull. Malays. Math. Sci. Soc. 42, No. 1, 83-99(2022).


\bibitem{khavison-1995} D. Khavinson, A note on entire solutions of the eiconal equation, Amer. Math. Mon. 102, 159-161(1995).



\bibitem{li-2005} B. Q. Li Eneire solutions of $(u_{z_1})^m+(u_{z_2})^n=e^g,$ Nagoya Math. J. Vol. 178, 151-162(2005).

\bibitem{li-2007}B. Q. Li, Entire solutions of eiconal type equations. Arch. Math. 89, no. 4, 350-357(2007).



\bibitem{li-2012}B. Q. Li, On Fermat-Type functional and partial differential equations, 209-222. I. Sahadini, D. C. Struppa (eds.), The Mathematical Legacy of Leon Ehrenpreis, Springer Proceedings in Mathematics 16, Springer-Verlag Italia 2012.

\bibitem{li-2014} B. Q. Li, On meromorphic solutions of generalized Fermat equations, International Journal of Mathematics
Vol. 25, No. 1 (2014) 1450002 (8 pages).

\bibitem{li-ye-2008} B. Q. Li, Z. Ye, On meromorphic solutions of $f^3+g^3=1,$ Arch. Math. 90, 39-43(2008).

\bibitem{liu-cao-cao-2012} K. Liu, T. B. Cao and H. Z. Cao, Entire solutions of Fermat-type differential-difference equations, Arch. Math. 99, 147-155(2012).



\bibitem{lv-li-2019} F. L\"{u} and Z. Li,  Meromorphic solutions of Fermat type partial differential equations, J. Math. Anal. Appl. 478, 864-873(2019).

\bibitem{magTal-2003} R. Magnanini and G. Talenti, On complex-valued solutions to a two-dimensional eikonal equation. II. SIAM J. Math. Anal. 34, no. 4., 805-835, (2003).

\bibitem{montel-1927} P.  Montel, Lecons sur les familles normales de fonctions analytiques et leurs
applications, Gauthier-Villars, Paris,, 135-136(1927).


\bibitem{nouguchiWinkelmannBook-2010} J. Noguchi and J. Winkelmann,  Nevanlinna Theory in Several Complex Variables and Diophantine Approximation, Springer, 2010.


\bibitem{ru-book-2001} M. Ru, Nevanlinna Theory and Its Relation to Diophantine Approximation, World Scientific Publishing Co., Singapore, 2001.

\bibitem{saleeby-1999}E. G. Saleeby, Entire and meromorphic solutions of Fermat type partial differential equations, Analysis 19, 69-376(1999).


\bibitem{shabat-book-1992} B. V. Shabat, Functions of Several Variables, Introduction to Complex Analysis, Part II, Translation Mathematical Monographs, Vol. 110 (American Mathematical Society, Providence, RI, 1992).



\bibitem{stoll-1968}W. Stoll, About entire and meromorphic functions of exponental type, Proc. Sympos. Pure Math., vol. 11, Amer. math. Soc., Providence, R. I., pp. 392-430(1968).



\bibitem{wiles-1995} A. Wiles,  Modular elliptic curves and Fermat's last theorem, Ann. Math. 141, 443-551(1995).


\bibitem{xu-cao-2020} L. Xu and T. B. Cao, Solutions of complex Fermat-type partial difference and differential-difference equations, Mediterr. J. Math. 15:227, 1-14(2018)/ Correction: Mediterr. J. Math. 17:8, 1-4(2020).

\bibitem{xuTuWang-2021} H. Y. Xu, J. Tu and H. Wang, Transcendental entire solutions for several Fermat type PDEs and PDDEs with two complex variables,  Rocky Mountain J. of Math. 51, No. 6, 2217-2235, (2021).

\bibitem{xu-liu-li-2020} H. Y. Xu, S. Y. Liu and Q. P. Li, Entire solutions for several systems of nonlinear difference and partial differential-difference equations of Fermat-type, J. Math. Anal. Appl. 483, No. 2, pp 20(2020).

\bibitem{xuXu-2022} H. Y. Xu and L. Xu, Transcendental entire solutions for several quadratic binomial and trinomial PDEs with constant coefficients, Analysis and Math. Phy., 12:64, pp 21(2022).


\bibitem{yezhuan-1995} Z. Ye, On Nevanlinna's second main theorem in projective space, Invent. Math. 122, 475-507(1995).

\bibitem{yezhuan-1996} Z. Ye, A sharp form of Nevanlinna's second main theorem of several complex variables, Math. Z. 222, 81-95(1996).

\bibitem{zhangXiaoFang-2022} M. Zhang, J. Xiao and M. Fang, Entire solutions for several Fermat type differential difference equations, AIMS Mathematics, 7(7), 11597-11613, (2022).

\bibitem{zhengXu-2022} X. M. Zheng and H. Y. Xu, Entire solutions for some Fermat type functional equations concerning difference and partial differential in $\C^2$, Analysis Math, 48 (1), 199-226, (2022).

\end{thebibliography}
\end{document}